\documentclass{amsart}
\usepackage{graphicx}
\usepackage{amsmath}
\usepackage{mathtools}
\mathtoolsset{showonlyrefs=true}
\usepackage{amsthm}
\usepackage{amssymb}
\usepackage{tikz-cd}
\usepackage{etoolbox}
\usepackage{enumitem}
\usepackage[hidelinks]{hyperref}
\usepackage[hyphenbreaks]{breakurl}
\usepackage[backend=bibtex]{biblatex}
\usepackage{tikz}
\usetikzlibrary{calc}

\newtheorem{thm}{Theorem}[section]
\newtheorem{theorem}[thm]{Theorem}
\newtheorem{corollary}[thm]{Corollary}

\newtheorem{lemma}[thm]{Lemma}
\newtheorem{prop}[thm]{Proposition}

\newtheorem{conjecture}[thm]{Conjecture}

\theoremstyle{definition}
\newtheorem{defn}[thm]{Definition}

\newtheorem{remark}[thm]{Remark}

\renewcommand{\O}{\mathcal{O}}

\newcommand{\Ce}{\mathbb{C}}
\newcommand{\BZ}{\mathbb{Z}}

\title[Modular reduction of representations of finite reductive groups]{Modular reduction of complex representations of finite reductive groups}
\date{March 30, 2026}
\author[R. Bezrukavnikov]{Roman Bezrukavnikov}
\address{Massachusetts Institute of Technology, Department of Mathematics, Cambridge, MA, 02139, USA}
\email{bezrukav@math.mit.edu}

\author[M.Finkelberg]{Michael Finkelberg}
\address{Einstein Institute of Mathematics, The Hebrew University of Jerusalem,
  Edmond J. Safra Campus, Giv’at Ram, Jerusalem, 91904, Israel;
\newline  National Research University Higher School of Economics}
\email{fnklberg@gmail.com}

\author[D. Kazhdan]{David Kazhdan}
\address{Einstein Institute of Mathematics, The Hebrew University of Jerusalem,
  Edmond J. Safra Campus, Giv’at Ram, Jerusalem, 91904, Israel}
\email{kazhdan@mail.huji.ac.il}

\author[C. Morton-Ferguson]{Calder Morton-Ferguson}
\address{Stanford University, Department of Mathematics, Stanford, CA, 94304, USA}
\email{caldermf@stanford.edu}

\begin{document}

\begin{abstract}

The main result describes the Brauer-Nesbitt reduction of unipotent representations of a finite group of Lie type, expressing it as an explicit linear combination of the restriction
of Weyl modules from the algebraic group to the group of $\mathbb{F}_q$ points. 
This partly confirms Lusztig's conjecture (2021), which was the main source of motivation 
for this work.

The explicit virtual representations of the algebraic group come from a certain endomorphism
of the space ${\mathbb Z}[T]$ of regular functions on the torus which approximates pullback under Frobenius and is linear over the ring
${\mathbb Z}[T]^W$ of $W$-invariant functions. This endomorphism is constructed from a new basis
for ${\mathbb Z}[T]$ over ${\mathbb Z}[T]^W$ which we call the Kazhdan-Lusztig-Steinberg basis.

We compare this basis to the canonical basis appearing in the study of modular representations of the algebraic group and  the related noncommutative Springer resolution. This leads
to canonically defined objects in the derived category of $G$-modules representing the above virtual representations and to  a geometric interpretation for the resulting lift of the principal series representation $\overline{\mathbb{F}_q} [G/P(\mathbb{F}_q)]$ to a virtual representation of the algebraic group, which comes from a decomposition of diagonal in the equivariant Grothendieck group of the partial flag variety.
\end{abstract}

\maketitle

\setcounter{tocdepth}{1}
\tableofcontents

\section{Introduction}

Given a complex representation $\rho$ of a finite group $\Gamma$, one can consider the \emph{reduction mod $p$} of this representation, an associated representation $\underline{\rho}$ over the field $\overline{\mathbb{F}}_p$. This operation was defined by Brauer and Nesbitt in \cite{BN}. In \cite{L}, Lusztig studied this operation in the case where $\Gamma = G(\mathbb{F}_p)$, a finite reductive group over a field of characteristic $p$. He presented a conjecture toward the goal of giving an explicit formula for the character of a representation obtained by applying this operation to any irreducible unipotent representation of $G(\mathbb{F}_p)$.

To formulate this conjecture, let $\mathcal{J}$ be the set of \emph{near involutions} of the Weyl group $W$, i.e.\ elements $w$ such that $w$ and $w^{-1}$ lie in the same left Kazhdan-Lusztig cell of $W$. For $W$ of classical type, this is exactly the set of involutions. Lusztig then defined for any $w \in \mathcal{J}$ a certain complex character $R_{\alpha_w}$ of $G(\mathbb{F}_q)$, each of which is a linear combination of characters of irreducible unipotent representations with coefficients defined in terms of the \emph{asymptotic Hecke algebra} associated to $W$ (which we recall in more detail in Section \ref{sec:asymptotic}). Finally, Lusztig proposed

\begin{conjecture}[Conjecture 2.3 in \cite{L}]\label{conj:lusztig}
\label{part:main}For every $w \in \mathcal{J}$, there exists a nonzero character $M_w$ of $G(\mathbb{F}_p)$ over $\overline{\mathbb{F}}_{p}$ such that for any irreducible unipotent representation $\rho$, we have
\begin{equation}
    \underline{\rho} = \sum_{w \in \mathcal{J}} (\rho : R_{\alpha_w})M_w.\label{eqn:formulaone}
\end{equation}
\end{conjecture}
No specific formula for $M_w$ was proposed in general, but these $M_w$ were further conjectured to satisfy some desirable properties, most importantly being in a sense ``independent of $p$," which we explain below. In Section 2 of \cite{L}, Lusztig computed the elements $M_w$ in Types $A_1, A_2, B_2, G_2$, and $A_3$.

In the present paper, our main construction is an explicit definition of $M_w$. Let $T$ be a torus for $G$ split over $\mathbb{F}_p$ and let $W$ be the Weyl group with longest element $w_0$. We study the module $\mathbb{C}[T]$ over $\mathbb{C}[T]^W$ and construct a system of generators $\{f_w\}_{w \in W}$ with a simple and explicit definition which satisfy many useful properties, which we call the Kazhdan-Lusztig-Steinberg basis, as it shares
some properties both with the canonical bases and with the Steinberg basis of ${\mathbb Z}[W]$
from \cite{St}.
There is a $\mathbb{C}[T]^W$-linear endomorphism $[q]^*$ of $\mathbb{C}[T]$ given on this basis by $[q]^*f_w(t) = f_w(t^q)$. We consider also a standard $\mathbb{C}[T]^W$-linear pairing $\langle, \rangle$ on $\mathbb{C}[T]$ given by the Weyl character formula, obtaining the following result.

\begin{theorem}\label{thm:main}
Conjecture \ref{conj:lusztig} is true, and it holds also for $G(\mathbb{F}_q)$ where $q$ is a power of $p$. Further, one can explicitly compute $M_w$ (and therefore $\underline{\rho}$), by defining
\begin{align}
    M_w & = \langle [q]^*f_{w_0d}, f_{w_0w}^*\rangle
\end{align}
for any $w \in W$, where $d$ is the unique Duflo involution in the same left Kazhdan-Lusztig cell as $w$.
\end{theorem}
We recall the definition of a Duflo involution in Section \ref{sec:defmw}. In Types $A_1$, $A_2$, $B_2$, $G_2$, and $A_3$ considered in \cite{L}, the $M_w$ we define agree with the ones computed in loc.\ cit. To any $w \in W$, we associate a sum of fundamental weights $\varpi_{I(w)}$ determined by its right descent set, which we explain in Section \ref{sec:klsbprelim}. In \cite{L}, Lusztig also suggests some properties that one might expect the elements $M_w$ to satisfy.
\begin{conjecture}[\cite{L}]\label{conj:mwprops}
The elements $M_w$ satisfy the following properties.
\begin{enumerate}[label=\roman*)]
    \item For any $w \in \mathcal{J}$, $M_w$ is a linear combination of characters of Weyl modules $V_{\lambda}$ for $\lambda$ very close to $(q-1)\varpi_{I(w)}$.
    \item We can write $\dim(M_w) = P_w(q)$ where $P_w(t) \in \mathbb{Q}[t]$ is independent of $q$. There exists an involution $w \leftrightarrow \tilde{w}$ of $\mathcal{J}$ such that $t^\nu P_w(1/t) = \pm P_{\tilde{w}}(t)$ and $\mathcal{L}_D(\tilde{w}) = S\setminus \mathcal{L}_D(w)$ for all $w \in \mathcal{J}$, where $\mathcal{L}_D(w) \subset S$ is the left descent set of $w$.
    \item For $w \in \mathcal{J}$, let $c(w)$ be the value of Lusztig's $a$-function on the two-sided cell of $W$ containing $w$. Then $P_w(t) \in t^{c(w)}\mathbb{Q}[t]$, $P_w(t) \not\in t^{c(w)+1}\mathbb{Q}[t]$.
    \item For every $w \in \mathcal{J}$, $M_w$ is the character of an actual representation of $G(\mathbb{F}_q)$, i.e.\ is a nonnegative linear combination of irreducible characters.
\end{enumerate}
\end{conjecture}

We will show that property i) holds for the $M_w$ defined in the present paper. However, we show that in fact not all of the above properties hold, and provide explicit counterexamples in the cases where they fail. In particular, properties ii) and iv) are false in general for our elements $M_w$, and both fail to hold in Type $A_4$.

Despite this, in addition to providing counterexamples we give a natural explanation for the failure of these properties, having to do with a discrepancy between the dual basis $\{f_w^*\}_{w \in W}$ and a certain \emph{pseudo-dual} basis $f^w$ which we also define. This discrepancy only appears in types beyond those for which Lusztig explicitly computed the elements $M_w$ in \cite{L}. We view this (in combination with some uniqueness results we examine in Section \ref{sec:uniqueness}) as compelling evidence for why the properties in Conjecture \ref{conj:mwprops} were satisfied in Lusztig's examples but 
not in general.

Thus our definition of the elements $M_w$ and subsequent proof of the formula \eqref{eqn:formulaone} provide an explicit expansion of the Brauer reduction $\underline{\rho}$ of an irreducible unipotent representation $\rho$ of $G(\mathbb{F}_q)$ as a linear combination of Weyl modules restricted to 
$G(\mathbb{F}_q)$.

 The elements $M_w$ arise in the above argument as a computational tool for evaluating the trace of operators induced by $\tilde \phi\colon \O(T)\to \O(T)$, which can be thought of as an $\O(T)^W$-linear approximation to the Frobenius endomorphism of the ring $\O(T)$.
  In Section \ref{sec:categorification} we lift $M_w$ to naturally defined objects
  in the derived category of $G$-modules, and describe a geometric way to lift
  the principal series $G(\mathbb{F}_q)$-module $\overline{\mathbb{F}_q}[(G/P)(\mathbb{F}_q)]$ to a virtual representation of the algebraic group $G$ compatible with the above
  lift of an arbitrary unipotent representation. 
  
  A general observation (see Lemma \ref{lem:diagonal}) links
  such a lift for the module $\overline{\mathbb{F}_q}[X(\mathbb{F}_q)]$,
 where $X$ is a projective variety over $\mathbb{F}_q$ equipped with a $G$-action, 
 to decomposition of diagonal in the equivariant $K$-group $K^G(X^2)$. We show that 
   noncommutative Springer resolution yields such a decomposition for the partial flag variety,
   and check compatibility with the above, based on the relation of the noncommutative Springer
   resolution with the canonical bases.


It is tempting to speculate that the virtual representations $M_w$ should admit a categorification in 
a still stronger sense,  perhaps finding an object of motivic nature that
produces both a representation of $G(\mathbb{F}_q)$ with characteristic zero coefficients and a
module over the algebraic group. 

To facilitate computations, we have made available an implementation of the constructions in the present paper in a public repository \cite{github}. Readers can use it to reproduce all of the computations in the present paper and to experiment further with the constructions introduced here.

{\bf Acknowledgments.} The first author wishes to thank Bao Le Hung for drawing his attention to \cite{L}. We would like to thank Alexander Samokhin and Wilberd van der Kallen for helpful communications and for generously sharing their code related to the paper \cite{SVK}. The research of D.K.\ was partially supported by ERC grant 101142781, R.B.\ was partly supported by the NSF grant DMS-2101507. The collaboration of R.B., M.F.\ and D.K.\ was partly supported by the US-Israel BSF grant 3013008329. 

\section{Setup}

\subsection{Preliminaries}
Let $q$ be a power of a prime $p$ and let $\mathsf{k} = \overline{\mathbb{F}}_p$. Let $G$ be an almost simple simply connected linear algebraic group over $\mathsf{k}$ with a given splitting over $\mathbb{F}_q$; we suppose that $p$ is sufficiently large, i.e.\ a good prime for $G$. Let $B$ be a Borel subgroup of $G$ defined over $\mathbb{F}_q$, let $T$ be a maximal torus defined and split over $\mathbb{F}_q$, and let $X^*(T)$ be the group of characters $T \to \mathsf{k}^\times$.

Let $\mu$ be the group of roots of unity in $\mathbb{C}^\times$. We fix once and for all a surjective homomorphism $\psi \colon  \mu \to \mathsf{k}^\times$ whose kernel is the subset of $\mu$ consisting of elements whose order is a power of $p$. For any finite group $\Gamma$ and a representation $\rho$ of $\Gamma$ over $\mathbb{C}$, Brauer and Nesbitt in \cite{BN} defined an associated representation $\underline{\rho}$ of $\Gamma$ over $\mathsf{k}$, which we call the reduction of $\rho$ mod $p$. The representation $\underline{\rho}$ is characterized by the property that for any $g \in \Gamma$, the eigenvalues of $g$ on $\underline{\rho}$ are obtained by applying $\psi$ to the eigenvalues of $g$ on $\rho$.

Following \cite[Section 0.1]{L}, we let $\mathcal{R}G(\mathbb{F}_q)$ (resp.\ $\mathcal{R}_pG(\mathbb{F}_q)$) be the Grothendieck group of finite-dimensional representations of $G(\mathbb{F}_q)$ over $\mathbb{C}$ (resp.\ over $\mathsf{k}$), with $\mathcal{R}^+G(\mathbb{F}_q)$ (resp.\ $\mathcal{R}_p^+G(\mathbb{F}_q)$) the submonoid consisting of classes of actual representations. We use $\mathcal{R}_pG$ to denote representations of the algebraic group $G$ over $\mathsf{k}$. 

For any $\lambda \in X^*(T)^+$, the set of dominant weights, we let $V_{\lambda} \in \mathcal{R}_p^+G(\mathbb{F}_q)$ be the character of the Weyl module corresponding to $\lambda$, which is given by the Weyl character formula. For a $p$-restricted weight $\lambda$, we let $L_{\lambda}$ be the corresponding irreducible representation of $G(\mathbb{F}_q)$ as explained in \cite{L}. When $\lambda = n_1\varpi_1 + \dots + n_{r}\varpi_r$ for fundamental weights $\varpi_k$, we set $V_{n_1,\dots,n_r} = V_\lambda$, $L_{n_1,\dots,n_r} = L_\lambda$.

Let $W$ be the Weyl group of $G$, and let $S$ be its set of simple reflections, in bijection with the set $\Pi$ of simple roots. For any $w \in W$, define
\begin{align*}
    \mathcal{L}_D(w) & = \{s \in S~|~\ell(sw) < \ell(w)\},\\
    \mathcal{R}_D(w) & = \{s \in S~|~\ell(ws) < \ell(w)\},
\end{align*}
the left and right descent sets for $w$. We let $I(w) \subset \Pi$ be the set of simple roots $\alpha_s$ for which $s \in \mathcal{R}_D(w)$, and $\overline{I(w)}$ its complement in $\Pi$.

By definition, $T$ admits an action by the Weyl group $W \cong N_G(T)/T$. Let $A = \mathbb{C}[T]^W$, $N = \mathbb{C}[T]$, $A_{\mathbb{Z}} = \mathbb{Z}[T]^W$, $N_{\mathbb{Z}} = \mathbb{Z}[T]$. Throughout, we will consider $N$ as an $A$-module and $N_{\mathbb{Z}}$ as an $A_{\mathbb{Z}}$-module. 

Let $[q] \colon  T \to T$ be given by $z \mapsto z^q$, let $\pi \colon  T \to T/W$ be the projection and $\overline{[q]} \colon  T/W \to T/W$ be the unique morphism such that $\pi \circ [q] = \overline{[q]} \circ \pi$. Let $(T/W)_q$ be the fixed point scheme of $\overline{[q]}$, and note that it is reduced. Let $\iota_q$ be the inclusion map $\iota_q \colon  (T/W)_q \hookrightarrow T/W$. The sheaf $\pi_*(\mathcal{O}_T)$ carries an action of $W$ while $\iota_q^*\pi_*(\mathcal{O}_T)$ carries an automorphism $\phi$ defined by $\phi(f) = [q]^*f$. For convenience, we write $A_q = \mathcal{O}((T/W)_q)$.

We will also make frequent use of the canonical basis $\{C_w\}_{w \in W}$ introduced by Kazhdan-Lusztig in \cite{KL} and the dual canonical basis $\{C_w'\}_{w \in W}$ for $\mathbb{Z}[W]$. We recall that for any $s \in S$,
\begin{align*}
    C_s & = s - 1, & C_s' & = s + 1.
\end{align*}
Let $i$ be the  anti-involution of $\Ce[W]$ sending $w$ to $w^{-1}$ and let $\sigma\colon \Ce[W]\to \Ce[W]$ be given
 by $\sigma(w)=(-1)^{\ell(w)} w$. Then $C_w' = \sigma(C_w)$ for any $w \in W$.

The basis $C_w$ was used in \cite{KL} to define a partition of the set $W$ into left-, right-, or two-sided Kazhdan-Lusztig cells. Any such set of Kazhdan-Lusztig cells admits a partial order $\leq$. For any two-sided cell $c$, we write
\begin{align*}
    \mathbb{Z}[W]_{\leq c} & = \mathrm{span}_{\mathbb{Z}}\{C_w~|~w \in c'\leq c\},& \mathbb{Z}[W]^{\leq c} & = \mathrm{span}_{\mathbb{Z}}\{C_w'~|~w \in c'\leq c\}, \\
\mathbb{Z}[W]_{< c} & = \mathrm{span}_{\mathbb{Z}}\{C_w~|~w \in c' < c\},& \mathbb{Z}[W]^{< c} & = \mathrm{span}_{\mathbb{Z}}\{C_w'~|~w \in c' < c\}, \\
\mathbb{Z}[W]_c & = \mathbb{Z}[W]_{\leq c}/\mathbb{Z}[W]_{< c} & \mathbb{Z}[W]^c & = \mathbb{Z}[W]^{\leq c}/\mathbb{Z}[W]^{< c}.
\end{align*}
Since $\sigma(C_w) = (-1)^{\ell(w)}C_{w}'$, we have that $\mathbb{Z}[W]_{c}$ is spanned by $C_{w}$, $w \in c$, and $\mathbb{Z}[W]^c$ is spanned by $C_w'$, $w \in c$. The $W$-module $\mathbb{Z}[W]^c$ is obtained from $\mathbb{Z}[W]_c$ by tensoring with the sign representation.

\subsection{Unipotent representations and almost-characters}

For any $\chi \in \mathrm{Irr}(W)$, there is a corresponding complex irreducible unipotent representation $U_{\chi}$ of $G(\mathbb{F}_q)$. The representations of the form $U_{\chi}$ for some $\chi \in \mathrm{Irr}(W)$ are called irreducible unipotent principal series representations. One can also associate to $\chi$ an element $V_\chi \in \mathcal{R}G(\mathbb{F}_q)$ called an \emph{almost character}, the virtual unipotent representation of $G(\mathbb{F}_q)$ over $\mathbb{C}$ whose character is the trace of the characteristic function of the corresponding Frobenius-invariant character sheaf on $G$, as introduced in \cite{LBook}.

In \cite[Section 12.1]{LBook}, Lusztig also introduced the non-abelian Fourier transform for representations of $G(\mathbb{F}_q)$. It can be thought of as an endomorphism $\mathrm{FT}$ of $\mathcal{R}G(\mathbb{F}_q)$ satisfying the property that for any $\chi \in \mathrm{Irr}(W)$, it exchanges $V_{\chi}$ with the character of $U_{\chi}$.

\subsection{The Brauer character}\label{sec:brauer}

The semisimple conjugacy classes of $G$ are in bijection with points of $T/W$, while the semisimple conjugacy classes of $G(\mathbb{F}_q)$ are in bijection with points of $(T/W)_q$.

By the map $\psi$, one can then view the Brauer character map $\beta$ as a map whose target is $\mathcal{O}((T/W)_q) = A_q$, yielding the following result.

\begin{prop}
    The Brauer character map $\beta$ defines an isomorphism 
    \begin{align}
        \mathcal{R}_pG(\mathbb{F}_q)\otimes_{\mathbb{Z}} \mathbb{C} \cong A_q.
    \end{align}
\end{prop}

\begin{prop}
    Let $\underline{V}_\chi$ be the reduction of $V_\chi$ modulo $p$. Then
    \begin{align}
        \beta([\underline{V}_\chi]) = \frac{1}{\dim \chi}\mathrm{tr}(\phi, [\iota_q^*\pi_*(\mathcal{O}_T) : \chi])
    \end{align}
\end{prop}
\begin{proof}
    This is an equivalent form of formula (1) in \cite[Section 3.2]{J} for $w = 1$.
\end{proof}

\begin{corollary}\label{cor:comm}
Suppose that $\tilde \phi$ is an automorphism of $\pi_*(\O_T)$ commuting with the action of $W$ and such that $\iota_q^*\tilde \phi =\phi$.
Let $\tilde V_\chi=\frac{1}{\dim \chi}\mathrm{tr}(\tilde \phi,  [\pi_*(\O_T):\chi]) \in A\cong \mathcal{R}_pG \otimes_{\mathbb{Z}}\mathbb{C}$, then 
\begin{equation}\label{barV}
\underline{V}_\chi=\tilde V_\chi|_{G(\mathbb{F}_q)}.
\end{equation}
\end{corollary}

To define $\tilde{\phi}$, and therefore $\tilde{V}_{\chi}$, we will proceed as follows. Given a basis $f_w$ for the free $A$-module $N$, one gets a uniquely defined $A$-module endomorphism such
that $\tilde \phi(f_w)=[q]^*f_w$.
If the $\mathbb{C}$-span of  $f_w$ is $W$-invariant then $\tilde{\phi}$ commutes with the action of $W$, so Corollary \ref{cor:comm} applies. Notice that $\tilde \phi$ only depends on the $\mathbb{C}$-span of $f_w$.

\section{The Kazhdan-Lusztig-Steinberg basis}

\subsection{Preliminaries}\label{sec:klsbprelim}

For $i \in \Pi$, let $\varpi_i$ be the corresponding fundamental weight. For any subset $I\subseteq \Pi$, we write $\overline{I}$ for its complement in $\Pi$. Let $\varpi_I=\sum_{i\in I}\varpi_i$. We can consider the parabolic subgroup $W_I \subset W$ generated by $\{s_i\}_{i \in I}$. Recall that for any $w$, 
\[I(w) = \{i \in \Pi~|~\ell(ws_i) < \ell(w)\}.\]

In \cite{St}, a basis $\{e_w\}_{w \in W}$ for $N = \mathbb{C}[T]$ over $A = \mathbb{C}[T]^W$ was constructed, which we refer to as the \emph{Steinberg basis}. Explicitly, for any $w \in W$, we have
\begin{equation*}
    e_w = w \exp(\varpi_{I(w)}).
\end{equation*}

This basis does {\em not} satisfy the desired property explained in Section \ref{sec:brauer} requiring that the $\mathbb{C}$-span of the basis be $W$-invariant. We will now describe a modification of this basis which will satisfy a weaker version of this $W$-invariance condition which will be sufficient for our purposes.

\subsection{The Kazhdan-Lusztig-Steinberg basis}\label{sec:klstein}

We now introduce our main construction, which is a different basis $\{f_w\}_{w \in W}$ for $N$ as an $A$-module. Since our goal is to provide formulas for $\mathrm{tr}(\tilde{\phi}, [\pi_*(\O_T) : \chi])$ for some $\tilde{\phi}$ defined in terms of this new basis, we will use a version of Lusztig's two-sided cell filtration on $N$, as every irreducible representation $\chi$ of $W$ appears with nontrivial multiplicity in some subquotient under this filtration. As a result, we will relax the condition of $W$-invariance for the $\mathbb{C}$-span of our basis, instead requiring only an analogue of this property for graded pieces with respect to this filtration.

Explicitly, suppose that we are given some partial order $\preceq$ on $W$ and some basis $\{f_w\}_{w \in W}$ for $N$ as an $A$-module. For any $w \in W$, let $N_{\preceq w}$ and $N_{\prec w}$ be the $A$ modules generated by $\{f_y\}_{y \preceq w}$ and $\{f_y\}_{y \prec w}$ respectively, with $\mathrm{gr}_w(N) = N_{\preceq w}/N_{\prec w}$. For any $y \preceq w$, we write $\bar f_y$ for the image of $f_y$ in $\mathrm{gr}_w(N)$.

Recall that with such a basis $\{f_w\}_{w \in W}$ in hand, we define the endomorphism $\tilde{\phi}$ of $N$ by
\begin{equation*}
    \tilde{\phi}(f_w) = [q]^*f_w.
\end{equation*}
One can then easily check the following.

\begin{lemma}\label{lem:fil}
Suppose that the following conditions are satisfied:
\begin{enumerate}[label=(\roman*)]
\item For every $w\in W$ the submodules $N_{\preceq w}$ and $N_{\prec w}$ are $W$-invariant.\label{enum:submodules}

\item For any $w \in W$, the $\mathbb{C}$-span of $\{\bar f_y\}_{y \preceq w}$ is a $W$-invariant subspace.\label{enum:winvariant}

\item For every irreducible representation $\chi$ of $W$ there exists a unique equivalence class of $w\in W$ such that $\chi$ appears in $gr_w(N)$ with nonzero multiplicity.\label{enum:irreps}
\end{enumerate}
Then the endomorphism $\tilde \phi$ preserves $N_{\preceq w}$ and $N_{\prec w}$, and the induced automorphism of $\mathrm{gr}_w(N)$ commutes with the action of $W$. Further,
\begin{equation*}
    \mathrm{tr}(\tilde{\phi}, [\mathrm{gr}_w(N) : \chi]) = \mathrm{tr}(\tilde{\phi}, [N : \chi])
\end{equation*}
\end{lemma}

This means once we define a basis satisfying the conditions in Lemma \ref{lem:fil}, if we set $\tilde V_\chi=\frac{1}{\dim \chi}\mathrm{tr}(\tilde \phi,  [\mathrm{gr}_w(N):\chi])$ for the unique equivalence class of $w$ for which $[\mathrm{gr}_w(N):\chi]\ne 0$, then the equation
\begin{equation}
    \underline{V}_\chi = \tilde{V}_{\chi}|_{G(\mathbb{F}_q)}
\end{equation}
introduced in Corollary \ref{cor:comm} still holds.
 
We now proceed to define a basis with the above properties. From now on, we use the partial order $\leq_{\mathrm{LR}}$ on $W$ given by Lusztig's two-sided cell order, so $w \leq_{\mathrm{LR}} y$ if and only if $w \in c, y \in c'$ for $c, c'$ two-sided cells satisfying $c \leq c'$. We write $w \leq_{\mathrm{L}} y$ for the order corresponding to left cells rather than two-sided cells, and we write $w \sim_\mathrm{L} y$ when $w, y$ are in the same left cell.

Let $(\ ,\ )$  be the standard pairing on $\Ce[W]$ given by \[(a,b)=\mathrm{tr}(a\cdot i(b), \Ce[W]).\] Let $C^w$ be the dual basis to $C_w$ with respect to this pairing,
thus $C^w=w_0 C_{w_0w}'$. 
\begin{defn}\label{def:basis}
For any $w \in W$, let
\begin{align*}
    f_w & =\frac{1}{|W_{I(w)}|}C_w' (\exp(\varpi_{\overline{I(w)}})),\\
    f^w & =\frac{1}{|W_{\overline{I(w)}}|} C^w(\exp(\varpi_{I(w)})).
\end{align*}
\end{defn}
Note that for any $w \in W$, $f_w \in N_\mathbb{Z}$. This is because one can write $C_w' = C_w''\cdot (|W_{I(w)}|a_{I(w)})$ for some $C_w'' \in \mathbb{Z}[W]$, where $a_{I(w)} = \frac{1}{|W_{I(w)}|}\sum_{w \in W_{I(w)}} w$ is the idempotent of the trivial representation of the parabolic subgroup $W_{I(w)}$. The claim then follows from the fact that $a_{I(w)}\exp(\varpi_{\overline{I(w)}}) = \exp(\varpi_{\overline{I(w)}})$.

\begin{defn}
    Let $\langle, \rangle \colon  N\times N\to A$ be the perfect pairing defined by the Weyl character formula, i.e. 
\begin{equation}
    \langle f, g\rangle =\frac{1}{\delta}\left(\sum_{w\in W} \mathrm{sgn}(w) w(fg)\right)
\end{equation}
where $\delta$ is the Weyl denominator. We clearly have
\begin{equation}\label{iot}
\langle af, g\rangle = \langle f, \sigma i(a)g\rangle
\end{equation}
for $a\in \Ce[W]$, $f,g\in N$.
\end{defn}
 
For $w \in W$, let $h(w) = (\rho, \varpi_{I(w)})$.
\begin{lemma}\label{lem:gram}
    For $w, v \in W$ with $h(v) \leq h(w)$,
    \[\langle f_w, f^v\rangle = \delta_{w,v}.\]
\end{lemma}
\begin{proof}
    One can check that if $h(v) < h(w)$, then $x\varpi_{\overline{I(w)}} + y\varpi_{I(v)}$ is a singular weight (i.e.\ has nontrivial stabilizer under $W$) for any $x, y \in W$ (note that $\overline{I(w)} = I(w_0w)$). This also holds if $h(v) = h(w)$ and $I(v) \neq I(w)$. It then follows from the Weyl character formula that $\langle f_w,f^v\rangle=0$ if $h(v)\leq h(w)$ and $I(v)\ne I(w)$. 
    
    Finally, if $h(v) = h(w)$ and $I(v) = I(w)$, then the pairing in question coincides with the pairing of the (dual) Kazhdan-Lusztig basis elements in $\mathbb{Z}[W/W_{I(w)}]$. This follows from the fact that
    \begin{align*}
        \langle w\exp(\varpi_{I(w)}), \exp(\varpi_{\overline{I(w)}})\rangle & = \begin{cases}
            0 & w \not\in W_{I(w)}W_{\overline{I(w)}}\\
            (-1)^x & w = xy, x \in W_{I(w)}, y \in W_{\overline{I(w)}},
        \end{cases}
    \end{align*}
    and therefore
    \begin{align*}
        \langle f_w, f^v\rangle = \delta_{w,v}
    \end{align*}
    in the case when $h(v) = h(w)$ and $I(v) = I(w)$.
\end{proof}

Lemma \ref{lem:gram} implies that the matrix $\langle f_w, f^v\rangle$ is upper triangular with ones along the diagonal when the basis is ordered according to some total order refining the partial order $\leq_{\mathrm{LR}}$ on $W$. This implies the following corollary.
\begin{corollary}
    The set $\{f_w\}_{w \in W}$ is a basis of $N_{\mathbb{Z}} = \mathbb{Z}[T]$ over $A_{\mathbb{Z}} = \mathbb{Z}[T]^W$.
\end{corollary}

In what follows, for convenience, we will say that $w \leq c$ for $w \in W$ and $c$ a two-sided cell if $w\in c'$ with $c' \leq c$.

\begin{defn}\label{def:mc}
    For a two-sided cell $c$, let
    \begin{align}
        N^{\leq c}_\mathbb{Z} & = \mathbb{Z}[W]^{\leq c} N_\mathbb{Z} = \oplus_{y \leq c} A_{\mathbb{Z}}f_y,\\
        N^{< c}_\mathbb{Z} & = \mathbb{Z}[W]^{< c} N_\mathbb{Z} = \oplus_{y < c} A_{\mathbb{Z}}f_y,\\
        N^c_\mathbb{Z} & = N^{\leq c}_\mathbb{Z}/N^{< c}_\mathbb{Z} = \oplus_{y \in c} A_{\mathbb{Z}}f_y.
    \end{align}
\end{defn}

The well-definedness of these submodules is guaranteed by the following lemma.
\begin{lemma}\label{lem:cellmz}
    Let $c$ be a two-sided cell for $W$. Then the following properties hold.
    \begin{enumerate}
        \item We have $\mathbb{Z}[W]^{\leq c}N_\mathbb{Z} = \oplus_{y\leq c}A_{\mathbb{Z}}f_y$.
        \item The $W$-saturation of $\bigoplus_{w \in c} \mathbb{Z}f_w$ lies in $N^{<c}_\mathbb{Z} \oplus (\bigoplus_{w \in c}\mathbb{Z}f_w)$.
        \item Define $\xi\colon  A[W]^c \to N_{\mathbb{Z}}^c$ by $\xi(C_w') = f_w$. Then $\xi$ is a $W$-equivariant isomorphism.\label{part:cellmz3}
    \end{enumerate}
\end{lemma}
\begin{proof}
    First note that $\bigoplus_{y\leq c}A_\mathbb{Z} f_y=\mathbb{Z}[W]^{\leq c}N$. Indeed,
  $\bigoplus_{y\leq c}A_\BZ f_y\subset\BZ[W]^{\leq c}N$ by definition of $f_y$.
  Choosing a direct complement $N=N'\oplus\bigoplus_{y\leq c}A_\BZ f_y$,
  we see that if $\bigoplus_{y\leq c}A_\BZ f_y\subsetneq\BZ[W]^{\leq c}N$, then
  $\BZ[W]^{\leq c}N$ has rank strictly bigger than $|\{y\in W\mid y\leq c\}|$. This implies that the same is true of $\BZ[W]^{\leq c}N/A_+$ as a subspace of $N/A_+ \cong H^\bullet(G/B)$ (where $A_+\subset A$ denotes the subspace generated by terms of positive degree) but this is a contradiction by the well-known structure of $H^\bullet(G/B)$ as a $W$-module.

  For any $v,w$ in a left cell $c_L\subset c$, we have $I(v)=I(w)$ by \cite[Section 8.6]{LUnequal}.
  For a fixed weight $\lambda$, the map $w\mapsto w\exp(\lambda)$ is
  clearly $W$-equivariant. It follows that the map $\BZ[W]^{c_L}\to N^{c_L}_{\mathbb{Z}}$ given by 
  $C'_w\mapsto f_w$, is $W$-equivariant. Here $\BZ[W]^{c_L}\subset\BZ[W]^c$
  is spanned by $C'_w$, $w\in c_L$. Further,
  $\BZ[W]^c=\bigoplus_{c_L\subset c}\BZ[W]^{c_L}$. It is then an isomorphism by part (1).
\end{proof}

Lemma \ref{lem:cellmz} and the general theory of two-sided cells for $W$ then imply the following.
\begin{prop}\label{prop:conditions}
    The basis $\{f_w\}_{w \in W}$ in Definition \ref{def:basis} satisfies the conditions of Lemma \ref{lem:fil} for the partial order $\leq_{\mathrm{LR}}$.
\end{prop}

\subsection{The dual basis and the pseudo-dual basis}

Let $\{f_w^*\}_{w \in W}$ be the dual basis to $\{f_w\}_{w \in W}$ under the perfect pairing $\langle, \rangle$.

Although it is tempting to conjecture that $f_w^* = f^w$ (perhaps modulo $N^{<c}_{\mathbb{Z}}$ for $c$ the two-sided cell containing $w$), this is only a small rank phenomenon. We checked the following proposition case-by-case in SageMath.
\begin{prop}\label{prop:dualpseudo}
    In Types $A_1$, $A_2$, $B_2$, $G_2$, $A_3$,
    \begin{equation}
        \bar f_w^* = \bar f^w \in \mathrm{gr}_wN.\label{eqn:dualispseudo}
    \end{equation}
    However, this is not true in general, and in particular fails to hold in Type $A_4$, for $w = s_2s_3s_2$.
\end{prop}

Of note for the purposes of studying Lusztig's conjectures is the fact that in all of the cases for which Lusztig computed the values $M_w$ in \cite{L}, the equation \eqref{eqn:dualispseudo} holds. This discrepancy between $f_w^*$ and $f^w$ in general type provides an explanation for the failure of \cite[Conjecture 2.3ii)]{L} for our elements $M_w$ as defined in Definition \ref{def:mw}; this is the subject of our Section \ref{sec:symmetry} to follow. That being said, Lemma \ref{lem:gram} proves useful in computing $\{f_w^*\}_{w \in W}$ in terms of $\{f^w\}_{w \in W}$ in practice, as one can apply the Gram-Schmidt algorithm to the pseudo-dual basis $\{f^y\}_{y \in W}$ to compute $f_w^*$. 

\subsection{Conclusion in Type $A$}

Together, Lemma \ref{lem:fil} and Proposition \ref{prop:conditions} yield formulas for $\tilde{V}_\chi$ when $\chi$ is an irreducible principal series representation. In the next section, we will show how to rephrase this fact in order to prove Theorem~\ref{thm:main}. Although for the general case we will need some more setup to state the final answer, we can already state it in the case where $G = SL(n)$.
\begin{corollary}
For $G=SL(n)$ and an irreducible representation $\chi$ of $W$ we have
$$\tilde V_\chi=\sum_y \langle [q]^*f_y,f_y^*\rangle$$
where the sum runs over all involutions in the 2-sided cell corresponding to $\chi$.
\end{corollary}
\begin{proof}
By Lemma \ref{lem:fil} and Proposition \ref{prop:conditions}, we have
\begin{align}
    \tilde{V}_\chi = \mathrm{tr}(\tilde{\phi}, [\mathrm{gr}_w(N) : \chi]),
\end{align}
where $w$ is some element of the unique two-sided cell for which $[\mathrm{gr}_w(N) : \chi] \neq 0$. Since we are in Type A, $\mathrm{gr}_w(N)$ breaks into a direct sum of left cell representations, each of which is isomorphic to $\chi$ as a $\mathbb{C}[W]$-module; further, each has the form $\mathbb{C}[W]f_y$ for some unique involution $y \sim_\mathrm{L} w$. Since $\tilde{\phi}$ intertwines the $W$-action, we see that
\begin{align}
    \frac{1}{\dim \chi}\mathrm{tr}(\tilde{\phi}, \mathrm{gr}_w(N)) & = \frac{1}{\dim \chi}\sum_y \mathrm{tr}(\tilde{\phi}, \mathbb{C}[W]f_y)\\
    & = \sum_y \langle [q]^*f_y, f_y^*\rangle
\end{align}
with the second line coming from Schur's lemma for $\tilde\phi$ restricted to the irreducible representation $\chi$, and where the sum runs over all involutions in the $2$-sided cell containing $w$.
\end{proof}
In the next section, we explain the proof of Theorem \ref{thm:main} in general type.

\section{Proof of Lusztig's conjecture}

\subsection{Definition of the elements $M_w$}\label{sec:defmw}

For any $w \in W$, let $\delta(w)$ be the degree of the Kazhdan-Lusztig polynomial $P_{1,w}$. It was shown in \cite{Jos} that for each left cell of $W$, the function $w \mapsto \ell(w) - 2\delta(w)$ reaches its minimum at a unique element of this cell, which is called the \emph{Duflo involution}. It is shown in \cite{LCellsII} that $\ell(w) - 2\delta(w) = c(w)$ for $c(w)$ in the present paper denoting the value on $w$ of Lusztig's $a$-function on the Weyl group introduced in \cite{LCellsI}, which provides an alternate characterization of Duflo involutions.

\begin{defn}\label{def:mw}
    For $w \in \mathcal{J}$, let 
    \begin{equation*}
        M_w = \langle [q]^*f_{w_0d}, f_{w_0w}^*\rangle,
    \end{equation*}
    where $d$ is the unique Duflo involution in the same left cell as $w$.
\end{defn}

\begin{lemma}\label{lem:mwweyl}
    For any $w \in \mathcal{J}$, there exists some subset $\Lambda_w \subset X^*(T)$ independent of $q$ and some coefficients $r_\lambda \in \mathbb{Z}$ such that
    \begin{equation}\label{eqn:mwlambdaw}
        M_w = \sum_{\lambda \in \Lambda_w} r_\lambda V_{(q-1)\varpi_{I(w)} + \lambda}.
    \end{equation}
\end{lemma}
\begin{proof}
    Let $w \in \mathcal{J}$, and let $d$ be the Duflo involution in the same left cell as $w$. First, note that $[q]^*f_{w_0d}$ is a linear combination of monomials from the $W$-orbit of $\exp(q\varpi_{I(w)})$, since recall that $\overline{I(w_0d)} = I(d) = I(w)$. The pairing of such a monomial $\exp(\nu)$ with another monomial $\exp(\mu)$ appearing with nonzero coefficient in $f_{w_0w}^*$ equals $V_{y(\nu + \mu) - \rho} = V_{(q-1)\varpi_{I(w)} + \lambda}$ for some $\lambda$, where $y \in W$ is chosen such that $y(\nu + \mu)$ is dominant. We can write
    \begin{align*}
        [q]^*f_{w_0d} & = \sum_{\nu \in W \cdot q\varpi_{I(w)}} a_\nu \exp(\nu),\\
        f_{w_0w}^* & = \sum_{\mu \in \Lambda_w'} b_\mu \exp(\mu),
    \end{align*}
    where $\Lambda_w'$ is independent of $q$, as it is defined as the set of $\mu$ such that $\exp(\mu)$ occurs with nonzero coefficient in $f_{w_0w}^*$.
    
    Then equation \eqref{eqn:mwlambdaw} holds, where $\Lambda_w$ is the set of all $\lambda$ for which $\mu + \nu$ and $(q-1)\varpi_{I(w)} + \rho + \lambda$ lie in the same orbit, where $\nu \in W\cdot q\varpi_{I(w)}$, $\mu \in \Lambda_w'$, and for every such $\lambda$,
    \begin{align}
        r_\lambda & = \sum_{\mu, \nu} \pm a_\nu b_\mu.
    \end{align}
\end{proof}

\subsection{Cell modules and the asymptotic Hecke algebra}\label{sec:asymptotic}

The basis $\{C_w\}_{w \in W}$ for $\mathbb{Z}[W]$ we consider in the present paper comes from the Kazhdan-Lusztig basis, which we still refer to as $C_w$, of the Hecke algebra $\mathcal{H}$ over $\mathbb{Z}[v, v^{-1}]$. We write $h_{x,y,z} \in \mathbb{Z}[v, v^{-1}]$ for the structure constants of multiplication in $\mathcal{H}$ with respect to the basis $\{C_w\}_{w \in W}$. Recalling again Lusztig's $a$-function defined in \cite{LCellsI} (which is constant on two-sided cells), we define for any $x, y, z \in W$ the integer $\gamma_{x,y,z} \in \mathbb{Z}$ to be the constant term of $v^{a(z)}h_{x,y,z^{-1}} \in \mathbb{Z}[v]$.

We now recall the definition of the asymptotic Hecke algebra $J$. Let $J$ be the free abelian group with basis $\{t_w\}_{w \in W}$, equipped with a ring structure given by
\begin{align}
    t_xt_y & = \sum_z \gamma_{x,y,z}t_{z^{-1}}.
\end{align}
This multiplication is associative, and $J$ has identity element $\sum_{d \in \mathcal{D}} t_d$ (here $\mathcal{D}\subset W$ is the set of all Duflo involutions). Given any two-sided cell $c$ of $W$, let $J_c$ be the two-sided ideal $\mathrm{span}\{t_w~|~w \in c\}$. We then have a direct sum decomposition $J = \oplus_{c} J_c$, where each $J_c$ has unit element given by
\begin{align}
    1_{J_c} & = \sum_{d \in \mathcal{D}\cap c} t_d.
\end{align}
One can show that there is a natural correspondence between irreducible modules over $\mathbb{C}[W]$, $\mathcal{H}$, and $J$, cf.\ for example the explanation given in \cite[Sections 8, 9]{WheresCurtis}. The following result appears in \cite{LCellsII}.

\begin{prop}[\cite{LCellsII}]\label{prop:bimodule}
    For a two-sided cell $c$, there exists a $(\mathbb{Z}[W]_c, J_c)$-bimodule $B_c$ (the ``regular" bimodule) with basis $b_w$, $w \in c$. It is defined so that the maps
    \begin{align}
        \mathbb{Z}[W]_c & \to B_c & C_w & \mapsto b_w\\
        J_c & \to B_c & t_w & \mapsto b_w
    \end{align}
    are each isomorphisms, as a left $\mathbb{Z}[W]_c$-module and as a right $J_c$-module respectively. The action of the algebras $\mathbb{Z}[W]_c$ and $J_c$ on $B_c$ are mutual centralizers for one another.
\end{prop}

The identification $\xi \colon  A[W]^c \to N_{\mathbb{Z}}^c$ from part (3) of Lemma \ref{lem:cellmz} then implies the following.
\begin{corollary}
The map
\begin{align*}
    \eta \colon  A[B_c] & \to N^{w_0c}\\
    t_w & \mapsto f_{w_0w}
\end{align*}
is an isomorphism of $A$-modules which intertwines the $W$-action after twisting by the sign representation. As a result, there is a well-defined action of $A[J_c]$ on $N^{w_0c}$.
\end{corollary}

In the formulation of \cite[Conjecture 2.3]{L}, Lusztig makes use of a virtual representation $R_{\alpha_w}$ defined in terms of the asymptotic Hecke algebra for any $w \in \mathcal{J}$. To recall it, we first recall the coefficients $c_{w, \chi}$ appearing in loc.\ cit.

\begin{defn}\label{def:raw}
For $c$ a two-sided cell of $W$, and $w \in c$,
\begin{enumerate}
    \item For any irreducible representation $\chi$ of $W$ occurring with nonzero multiplicity in $\mathbb{Z}[W]_c$, let $c_{w,\chi}$ be the trace of $t_w$ in the corresponding representation of $J_c$. 
    \item Let
    \begin{align*}
        R_{\alpha_w} & = \frac{1}{|W|} \sum_{\substack{\chi\in \mathrm{Irr}(W)\\ y \in W}} c_{w,\chi}\mathrm{tr}(y, \chi) R_y,\\
        & = \sum_{\chi \in \mathrm{Irr}(W)} c_{w,\chi} V_{\chi}
    \end{align*}
    where $R_y$ is the unipotent Deligne-Lusztig character associated to $y$ as defined in \cite[Section 1.5]{DL}.
\end{enumerate}
\end{defn}

\subsection{Proof of Lusztig's main conjecture}
In this section, we prove Conjecture~\ref{conj:lusztig}, in the form of Corollary \ref{cor:mainagain} below. 

\begin{prop}\label{prop:jring}
    For any two-sided cell $c$, the endomorphism $\tilde{\phi}|_{\mathrm{gr}_c(N)}$ coincides with the right action of the element $h \in A \otimes_{\mathbb{Z}} J_c$ given by
    \begin{align}
        h & = \sum_{w \in c} M_w t_w.
    \end{align}
\end{prop}
\begin{proof}
    We fix a cell $c$ and write $\tilde{\phi} = \tilde{\phi}|_{\mathrm{gr}_c(N)}$ for convenience. Since $\tilde{\phi}$ intertwines the $W$-action on $\mathrm{gr}_c(N)$, so does the endomorphism $\eta^{-1}\circ \tilde{\phi} \circ \eta$ of $B_c$. By Proposition~\ref{prop:bimodule}, this means $\eta^{-1}\circ \tilde{\phi} \circ \eta$ can be written as the right action of some element $h \in A \otimes J_c$.

    Since $\sum_{d \in \mathcal{D}(c)} t_d$ is the identity element of $J_c$, we have
    \begin{align*}
        h & = (\eta^{-1}\circ \tilde{\phi} \circ \eta)\left(\sum_{d \in \mathcal{D}(c)} t_d\right)\\
        & = (\eta^{-1}\circ \tilde{\phi})\left(\sum_{d \in \mathcal{D}(c)} f_{w_0d}\right)\\
        & = \eta^{-1}\left(\sum_{d \in \mathcal{D}(c)} \sum_{w \sim_\mathrm{L} d} \langle [q]^*f_{w_0d}, f_{w_0w}^*\rangle f_{w_0w}\right)\\
        & = \sum_{d \in \mathcal{D}(c)} \sum_{w \sim_\mathrm{L} d} M_w \eta^{-1}(f_{w_0w}) \\
        & = \sum_{w \in c} M_w t_w.
    \end{align*}
\end{proof}

\begin{corollary}\label{cor:mainagainv}
    For $\chi$ an irreducible representation of $W$,
    \begin{align}
        \underline{V}_\chi & = \sum_{w \in \mathcal{J}} (V_\chi, R_{\alpha_w})M_w.\label{eqn:principalseries}
    \end{align}
\end{corollary}
\begin{proof}
    Let $\chi$ be an irreducible representation of $W$. By Lemma \ref{lem:fil} and Proposition~\ref{prop:conditions}, we know that
    \begin{align}
        \underline{V}_{\chi} = \tilde{V}_{\chi}|_{G(\mathbb{F}_q)},
    \end{align}
    where $\tilde{V}_{\chi} = \frac{1}{\dim \chi}\mathrm{tr}(\tilde{\phi}, [\mathrm{gr}_w(N) : \chi])$. By Proposition \ref{prop:jring}, we then have
    \begin{align}
    \tilde{V}_{\chi}|_{G(\mathbb{F}_q)} & = 
        \frac{1}{\dim\chi}\mathrm{tr}\left(\sum_{w \in c} M_wt_w, [\mathbb{C}[W]_c : \chi]\right)\\
        & = \frac{1}{\dim \chi}\sum_{w \in c}\mathrm{tr}(t_w, [\mathbb{C}[W]_c : \chi])M_w \\
        & = \sum_{w \in c} c_{w,\chi}M_w,
    \end{align}
    by the definition of $c_{w, \chi}$. We then note that $c_{w,\chi} = (V_\chi, R_{\alpha_w})$ by the definition of $R_{\alpha_w}$.
\end{proof}

Now note that expanding both sides of \eqref{eqn:principalseries} linearly gives the same result when $V_\chi$ is replaced by any linear combination of almost characters. Recall that for any $w \in W$, $R_w$ lies in the span of almost characters $V_\chi$; further, the value of any irreducible unipotent character $\chi$ on any semisimple conjugacy class is determined by $(\chi, R_w)$ across all $w \in W$ cf.\ \cite[p.\ 383]{Car}. Since the Brauer reduction map vanishes on any class function on $G(\mathbb{F}_q)$ supported away from semisimple elements, this implies that Corollary \ref{cor:mainagainv} is true for any irreducible unipotent representation. Thus we arrive at the following resolution to Conjecture \ref{conj:lusztig}.
\begin{corollary}\label{cor:mainagain}
    For $\rho$ any irreducible unipotent representation of $G(\mathbb{F}_q)$,
    \begin{align}
        \underline{\rho} & = \sum_{w \in \mathcal{J}} (\rho : R_{\alpha_w})M_w.\label{eqn:allunipotent}
    \end{align}
\end{corollary}

\section{Properties of the virtual characters $M_w$}

\subsection{Bounding weights in the Weyl module expansion}

The second part of \cite[Conjecture 2.3]{L} asserts that the $V_{\lambda}$ which occur with nonzero coefficient in $M_w$ (as in Lemma \ref{lem:mwweyl}) are such that $\lambda$ is ``very close" to $\exp((q-1)\varpi_{I(w)})$. In this section, we prove this part of the conjecture and provide an explicit bound.

\begin{lemma}
    For any $w, y \in W$, $\langle f_w^*, (f^y)^*\rangle$ is a $\mathbb{Q}$-linear combination of terms of the form
    \begin{align}
        \langle f_{z_1}, f^{z_2}\rangle\langle f_{z_2}, f^{z_3}\rangle \dots \langle f_{z_{k-1}}, f^{z_k}\rangle,\label{eqn:sequence}
    \end{align}
    where $y = z_1, \dots, z_k = w$ is some sequence of elements of $W$ with $$h(y) = h(z_1) < h(z_2) < \dots < h(z_k) = h(w).$$
\end{lemma}
\begin{proof}
    This follows from Lemma \ref{lem:gram} and inverting the corresponding upper triangular matrix with entries $m_{wy} = \langle f_w, f^y\rangle$.
\end{proof}

\begin{prop}
For any $w \in W$, $M_{w_0w}$ can be expressed as a linear combination of terms of the form $\mathrm{ch}(V_{\lambda})$ where $\lambda = (q - 1)\varpi_{\overline{I(w)}} + \mu$, where $\mu$ is in the $W$-orbit of some dominant weight between $0$ and $2(h(w) + 1)\rho$.
\end{prop}

\begin{proof}
By Lemma \ref{lem:gram},  we have
\begin{align}
    f_w^* = f^w + \sum_{\substack{y\in W\\ h(y) < h(w)}} \langle f_w^*, (f^y)^*\rangle f^y,
\end{align}
so 
\begin{align}
    \langle [q]^*f_d, f_w^*\rangle - \langle [q]^*f_d, f^w\rangle = \sum_{\substack{y\in W\\ h(y) < h(w)}} \langle f_w^*, (f^y)^*\rangle \cdot \langle [q]^*f_d, f^y\rangle.\label{eqn:qdual}
\end{align}
By writing out $\langle f_w^*, (f^y)^*\rangle$ as in \eqref{eqn:sequence} and then expressing it as a linear combination of weights $\exp(\mu)$, those which can appear with nonzero coefficient are all in the $W$-orbit of some dominant weight between $0$ and $2(h(w) - h(y))\rho$. This is because there are at most $h(w) - h(y)$ terms in the expression, and each term is itself a linear combination of weights in the $W$-orbit of a dominant weight between $0$ and $2\rho$. In particular, the entire right-hand side of \eqref{eqn:qdual} is a linear combination of $\exp(\mu)$ for $\mu$ in the orbit of some dominant weight between $0$ and $2h(w)\rho$.

Now note that $\langle [q]^*f_d, f^w\rangle$ is a linear combination of $\exp(\mu)$ for $\mu$ in the $W$-orbit of weights of the form $q\varpi_{\overline{I(w)}} + y\varpi_{I(w)}$  for some $y \in W$. Adding the right-hand side of \eqref{eqn:qdual} to this term and writing $q\varpi_{\overline{I(w)}} = (q-1)\varpi_{\overline{I(w)}} + \varpi_{\overline{I(w)}}$ we get that $\langle [q]^*f_d, f_w^*\rangle$ can be written as a linear combination of $\exp(\mu)$ for
\begin{align}
    \mu  = z((q-1)\varpi_{\overline{I(w)}} + \mu')
\end{align}
for some $z \in W$, where $\mu'$ is in the $W$-orbit of some dominant weight between $0$ and $(2h(w) + 1)\rho$. For $q > 2h(w) + 1$, it is the only dominant element in its $W$-orbit, and so when we express it as a linear combination of characters $V_{\lambda}$, we must have that 
\begin{align}
    \lambda + \rho = (q - 1)\varpi_{\overline{I(w)}} + \mu',
\end{align}
where again $\mu'$ is in the $W$-orbit of some dominant weight between $0$ and $(2h(w) + 1)\rho$. This means $\lambda - (q - 1)\varpi_{\overline{I(w)}}$ is in the $W$-orbit of some dominant weight between $0$ and $2(h(w) + 1)\rho$.
\end{proof}

One can then use the following bound not depending on $w$.
\begin{corollary}\label{cor:finalbound}
    For any $w \in W$, $M_w$ can be expressed as a linear combination of terms of the form $\mathrm{ch}(V_\lambda)$, where $\lambda - (q - 1)\varpi_{\overline{I(w)}}$ lies in the $W$-orbit of some dominant weight less than or equal to $k\rho$, where $k = 2(\rho, \rho) + 2$.
\end{corollary}

\subsection{Conjectured symmetries for dimension polynomials}\label{sec:symmetry}

In this section, we show that the following conjecture is false for the $M_w$ defined in the present paper. To state it, write $\dim(M_w) = P_w(q)$ where $P_w(t) \in \mathbb{Q}[t]$ for $t$ an indeterminate. Lemma \ref{lem:mwweyl} shows that this polynomial is well-defined.
\begin{conjecture}[Conjecture 2.3ii) in \cite{L}]\label{conj:symmetry}
    There exists an involution $w \leftrightarrow \tilde{w}$ of $\mathcal{J}$ such that $t^\nu P_w(1/t) = \pm P_{\tilde{w}}(t)$ and $\mathcal{L}_D(\tilde{w}) = S\setminus \mathcal{L}_D(w)$.
\end{conjecture}

For the $M_w$ which were already computed in \cite{L} in types $A_1$, $A_2$, $B_2$, $G_2$, $A_3$, this conjecture holds. Despite our $M_w$ agreeing with those computations, these elements fail to satisfy this symmetry in Type $A_4$.
\begin{prop}\label{prop:nosymmetry}
    The property in Conjecture \ref{conj:symmetry} does not hold for $G = SL(5)$ using our definition of $M_w$. Choosing $w = s_2s_3s_2$ or $w = s_1s_2s_3s_4s_3s_2s_1$, there exists no associated $\tilde{w}$ satisfying the conjecture.
\end{prop}
\begin{proof}
    Using Definition \ref{def:mw}, one can use the Weyl dimension formula to compute
    \begin{align*}
        M_{s_2s_3s_2} & = V_{0,q-1,q-1,0} - V_{0,q-2,q-2,0},\\
        M_{s_1s_2s_3s_4s_3s_2s_1} & = V_{q-3,0,0,q-3} + 2V_{q-2,0,0,q-2} + V_{q-1,0,0,q-1},\\
        \dim(M_{s_2s_3s_2}) & = \frac{1}{36}q^3 + \frac{19}{36}q^5 + \frac{4}{9}q^7,\\
        \dim(M_{s_1s_2s_3s_4s_3s_2s_1}) & = \frac{17}{36}q^3 + \frac{17}{36}q^5 + \frac{1}{18}q^7.
    \end{align*}
    In Table \ref{tab:dimtable}, we list all of the values $P_w(q)$ for $w \in W$, and clearly none of them are of the form $\pm t^\nu P_w(1/t)$ when $P_w(t)$ is either of the two polynomials above.
\end{proof}

The following is a partial answer to Conjecture \ref{conj:mwprops}iii).
\begin{prop}
    For any $w \in W$, the polynomial $P_w(t) = \dim(M_w)$ is divisible by $t^{c(w)}$, where $c(w)$ is the value of Lusztig's $a$-function on the two-sided cell of $W$ containing $w$.
\end{prop}
\begin{proof}
    The evaluation at $1\in T/W$ is a surjective homomorphism
  $N\twoheadrightarrow N\otimes_A\mathbb{C}\cong H^\bullet(G/B)$. The endomorphism
  $[q]^*$ induces the endomorphism $[q]^*_1$ of $H^\bullet(G/B)$ that acts by
  $q^i$ on $H^{2i}(G/B)$. Since $M_w$ was a matrix entry of $[q]^*$ in
  the basis $\{f_y,\ y\in W\}$, the dimension of $M_w$ is a matrix entry
  of $[q]^*_1$ in the basis $\{(f_y)|_1,\ y\in W\} \subset H^\bullet(G/B)$. It is known that
  the $W$-representations in a cell $c$ appear in the cohomology $H^{2i}(G/B)$
  only for $i\geq a(c)$. Since $a(c')\leq a(c)$ for $c'\leq c$, the desired
  divisibility follows.
\end{proof}

\begin{remark}
    If one defines for $w \in \mathcal{J}$ the quantity
    \begin{align}
        \tilde M_w = \langle [q]^*f_{w_0d}, f^{w_0w}\rangle,
    \end{align}
    one can check that the symmetry predicted in Conjecture \ref{conj:symmetry} holds for the quantities $\dim(\tilde{M}_w)$. By Proposition \ref{prop:dualpseudo}, $M_{w_0s_2s_3s_2} \neq \tilde{M}_{w_0s_2s_3s_2}$ in Type $A_4$, which gives an explanation for why the symmetry property fails in this case.
\end{remark}

\subsection{Positivity} As we have demonstrated, Definition \ref{def:mw} allows $M_w$ to be explicitly computed in any type as a $\mathbb{Z}$-linear combination of terms $V_{\lambda}$, where each $V_{\lambda}$ is the class of a Weyl module in the Grothendieck group of $k$-representations of $G(\mathbb{F}_q)$. In this section, we consider the question of whether $M_w$ can be written as a nonnegative linear combination of classes of irreducible representations; namely, whether $M_w$ is the class of a genuine representation or merely a virtual representation.

\begin{prop}\label{prop:nopositivity}
    There exist examples of $M_w$ which do not lie in $\mathcal{R}_{p}^+G(\mathbb{F}_q)$. For example, in Type $B_3$ for $w = s_1$, and in Type $A_4$ for $w = s_2s_3s_2$.
\end{prop}
\begin{proof}
    We used the software \cite{dynkin} to compute the following expansions for $M_w$ in terms of the irreducible representations $L_{\lambda}$.

    In Type $B_3$, we have
    \begin{align*}
        M_{s_1} & = V_{q-4,0,0} - V_{q-3,0,0}\\
        & = L_{q-4,0,0} - L_{q-3,0,0}.
    \end{align*}

    In Type $A_4$, we have
    \begin{align*}
    M_{s_2s_3s_2} & = V_{0,q-1,q-1,0} - V_{0,q-2,q-2,0}\\
    & = (L_{0,q-1,q-1,0} + L_{0,q-2,q-2,0}) - (L_{0,q-2,q-2,0} + L_{q-2,0,0,q-2})\\
    & = L_{0,q-1,q-1,0} - L_{q-2,0,0,q-2}.
    \end{align*}
\end{proof}

For all the types for which $M_w$ was computed in \cite{L} (Types $A_1$, $A_2$, $B_2$, $G_2$, $A_3$), positivity is satisfied. For example, in Table \ref{tab:a3positivity} for Type $A_3$, we explicitly write each $M_w$ in terms of the basis $L_{\lambda}$ of irreducibles, exhibiting positivity in this case. We used the software \cite{dynkin} to produce this table.

\section{On uniqueness of the elements $M_w$}\label{sec:uniqueness}

In this section we explain in what sense our collection of elements $\{M_w\}_{w \in W}$ is unique in satisfying some of the properties conjectured in Conjecture \ref{conj:mwprops}. In particular, we show that if some collection $\{M_w'\}_{w \in W}$ defined ``in terms of $q$" (cf.\ Definition \ref{def:qfamily}) satisfies Conjecture \ref{conj:lusztig} and a precise version of property i) in Conjecture \ref{conj:mwprops}, then this collection is very close to our $\{M_w\}_{w \in W}$, in a sense which we make precise in Definition \ref{def:candidate}. The main result of this section is Corollary \ref{cor:cantdoit}, which shows that with the aforementioned assumptions in hand, properties ii) and iv) of Conjecture \ref{conj:mwprops} are impossible to satisfy.

In rough terms, in this section we hope to convey the idea that our collection $\{M_w\}_{w \in W}$ is the ``correct" collection to use in answering Conjecture \ref{conj:lusztig}, and that there exists no better answer which might satisfy more of the properties in Conjecture \ref{conj:mwprops} than the ones already enjoyed by the $\{M_w\}_{w \in W}$ defined in the present paper.

\begin{defn}\label{def:candidate}
    Let $\gamma$ be a dominant weight. We say that two weights $\lambda_1$ and $\lambda_2$ are $\gamma$-close if $\lambda_1 - \lambda_2$ is in the $W$-orbit of some dominant weight $\gamma' < \gamma$. 
    
    We say that $\{M_w'\}_{w \in \mathcal{J}}$ is a \emph{$\gamma$-close collection of candidate elements} if:
    \begin{itemize}
        \item For any irreducible unipotent representation of $G(\mathbb{F}_q)$,
        \begin{align}
            \underline{\rho} = \sum_{w \in \mathcal{J}} (\rho : R_{\alpha_w}) M_w'
        \end{align}
        \item For all $w \in W$, $M_w'$ is a linear combination of $V_{\lambda}$ for $\lambda$ being $\gamma$-close to $(q-1)\varpi_{I(w)}$.
    \end{itemize}
\end{defn}

For $\gamma = k\rho$, where $k = 2(\rho, \rho) + 2$, we have shown in Corollary \ref{cor:finalbound} and in our main result that $\{M_w\}_{w \in W}$ as defined in the present paper is a $\gamma$-close collection of candidate elements.

A natural question is whether, for any fixed dominant weight $\gamma$, the collection $\{M_w\}_{w \in W}$ defined in the present paper is the unique $\gamma$-close collection of candidate elements for sufficiently large $q$. In fact this is not precisely the case, as evidenced by the following example.
\begin{prop}
    Suppose $G = SL(n)$ for $n \geq 3$. Then there exists some $\gamma$ such that for any $q$, the collection $\{M_w\}_{w \in W}$ is not the unique $\gamma$-close collection of candidate elements.
\end{prop}
\begin{proof}
    We exhibit an example for $G = SL(3)$; it will then be clear that a similar construction is possible for any $SL(n)$, $n \geq 3$. First, we write 
    \[A = V_{1,1} + V_{0,0}.\]
    We note that $A = \mathrm{ch}(V_{1,0} \otimes V_{0,1})$, and that $L_{1,0} = V_{1,0}$, $L_{0,1} = V_{0,1}$. By the Steinberg tensor product formula, we then note that $A = L_{p,1} = L_{1,p}$. By writing this as a linear combination of Weyl modules in two different ways (cf.\ \cite{dynkin} for a computational method with which to do so), we see that
    \begin{align}
        A & = V_{p,1} - V_{p-2,2} + V_{p-4,0}\\
        & = V_{1,p} - V_{2,p-2} + V_{0,p-4}.
    \end{align}
    Let $\{M_w'\}_{w \in W}$ be defined by $M_w' = M_w$ for $w \notin \{s_1, s_2\}$, with
    \begin{align}
        M_{s_1}' & = M_{s_1} + A\\
        M_{s_2}' & = M_{s_2} - A.
    \end{align}
    Since for any two-sided cell $c$, $\sum_{w \in c} M_w' = \sum_{w \in c} M_w$, we see that the first condition in Definition \ref{def:candidate} is satisfied, since we know it is satisfied by $\{M_w\}_{w \in W}$. Finally, by the two expansions above for the element $A$, we see that $\{M_w'\}_{w \in W}$ is a $\gamma$-close collection of candidate elements for, say, $\gamma = 4\rho$. 
\end{proof}

The following result can be checked case-by-case. It is equivalent to the statement that any two semistandard Young tableaux of the same shape and with identical descent sets must also be equal, provided that these tableaux have at most five boxes.
\begin{lemma}\label{lem:tableau}
    Suppose that $G = SL(n)$ for $1 \leq n \leq 5$. If $c$ is some two-sided cell and $w, w' \in c$ are involutions, then $w = w'$ if and only if $I(w) = I(w')$.
\end{lemma}

\begin{prop}\label{prop:uniqueness}
    Suppose that $G = SL(n)$ for $n \leq 5$. Then for any fixed $\gamma$, there exists $q' > 0$ such that if $q > q'$ and $\{M_w'\}_{w \in W}$ is a $\gamma$-close collection of candidate elements, then for all $w \in W$, $M_w' - M_w$ is a linear combination of $V_\lambda$ for $\lambda$ which are $(\gamma + 2\rho)$-close to $0$.
\end{prop}
\begin{proof}
    Let $\gamma$ be a fixed dominant weight. It is clear that we can choose $q'$ large enough so that if $q > q'$, then for any distinct subsets $I, I' \subset \Pi$, there exist no weights $\lambda$ which are $(\gamma + 2\rho)$-close to both $(q-1)\varpi_{I}$ and $(q-1)\varpi_{I'}$.

    The fact that $\{M_w\}_{w \in W}$ and $\{M_w'\}_{w \in W}$ both satisfy equation \eqref{eqn:formulaone} for any unipotent representation implies that for any two-sided cell $c$,
    \begin{equation}\label{eqn:cancel}
        \sum_{w \in \mathcal{J} \cap c} (M_w' - M_w) = 0.
    \end{equation}
    For any $w \in W$, $M_w' - M_w$ is a linear combination of $V_{\lambda}$ for $\lambda$ being $\gamma$-close to $(q-1)\varpi_{I(w)}$. Such a linear combination can then be written as a linear combination of $L_{\lambda}$ for $\lambda$ which is $(\gamma + \rho)$-close to $(q-1)\varpi_{I(w)}$. By the Steinberg tensor product theorem, this can again be rewritten as a linear combination of $L_{\lambda}$ for $\lambda$ being $q$-restricted and $(\gamma + 2\rho)$-close to either $(q-1)\varpi_{I(w)}$ or to $0$.

    We now claim that in fact all $L_\lambda$ appearing in this linear combination with nonzero coefficient are of the second type, with $\lambda$ being $(\gamma + 2\rho)$-close to $0$. Suppose otherwise; then some $L_{\lambda}$ term for $\lambda$ which is $(\gamma + \rho)$-close to $(q-1)\varpi_{I(w)}$ would have to also appear with nontrivial coefficient in the analogous expansion for $M_y' - M_y$ for $y \neq w$, $y \in c$, so that these terms would cancel out in equation \eqref{eqn:cancel}. But by our construction of this expansion, all of the $L_{\lambda}$ terms in the expression for $M_y' - M_y$ must be $(\gamma + 2\rho)$-close to either $(q-1)\varpi_{I(y)}$ or to $0$. By our choice of $q'$, this means $\varpi_{I(y)} = \varpi_{I(w)}$ or $\varpi_{I(y)} = 0$. By Lemma \ref{lem:tableau}, this implies $y = w$, which is a contradiction.

    This means each $M_w' - M_w$ is a linear combination of $L_\lambda$ (and therefore of $V_{\lambda}$) for $\lambda$ being $(\gamma + 2\rho)$-close to $0$.
\end{proof}

\begin{defn}\label{def:qfamily}
    Consider a family of collections of elements $\{M_w\}_{w \in W}$ defined simultaneously across all large values of $q$; we write $M_w^q$ when clarity is necessary to specify the dependence of the definition of $M_w$ on $q$. For a fixed $w \in W$, we say that $M_w$ is a linear combination of Weyl characters \emph{depending on $q$} if there exist polynomials $f_i(t) \in \mathbb{Z}[t][X^*(T)]$ and $c_i \in \mathbb{Z}$ such that
    \[M_w^q = \sum_i c_iV_{f_i(q)}\]
    for all $q$ sufficiently large. In this situation we say its \emph{dimension polynomial} $P_w(t) \in \mathbb{Q}[t]$ is the unique polynomial such that $\dim(M_w^q) = P_w(q)$ for any fixed $q$.
\end{defn}

Proposition \ref{prop:uniqueness} then clearly implies the following. 
\begin{corollary}\label{cor:polysdiffer}
    Let $G = SL(n)$, $n \leq 5$. For any dominant weight $\gamma$, suppose there exists a $\gamma$-close collection of candidate elements $\{M_w'\}_{w \in W}$ each of which is a linear combination of Weyl characters depending on $q$. Then for every $w \in W$, the dimension polynomials of $M_w$ and $M_w'$ differ by a constant term.
\end{corollary}

\begin{corollary}\label{cor:cantdoit}
    For $G = SL(5)$, and any dominant weight $\gamma$, there exists no $\gamma$-close collection of candidate elements $\{M_w'\}_{w \in W}$ with each $M_w'$ a linear combination of Weyl characters depending on $q$ which satisfies properties ii) or iv) of Conjecture \ref{conj:mwprops}.
\end{corollary}
\begin{proof}
    We note that the failure of property iv), i.e.\ positivity, which is exhibited in the example given in Proposition \ref{prop:nopositivity}, cannot be resolved by an adjustment by any linear combination of $V_\lambda$ for $\lambda$ close to $0$, so long as $q$ is sufficiently large. So there exists no collection $\{M_w'\}_{w \in W}$ satisfying the assumptions of the corollary such that $M_{s_2s_3s_2}'$ satisfies positivity, since it must differ from $M_{s_2s_3s_2}$ by a linear combination of such $V_\lambda$.

    Corollary \ref{cor:polysdiffer} then shows that property ii), which describes a symmetry among dimension polynomials, cannot hold for the example in Proposition \ref{prop:nosymmetry}, since the failure of symmetry in this case cannot be resolved by adjusting the polynomials in question by a constant.
\end{proof}

\section{Categorical interpretation by coherent sheaves on $G/P$}\label{sec:categorification}

We now explain a perspective by which the constructions in the present paper can be understood in terms of coherent sheaves and derived categories on flag varieties. We define a collection of objects in $D^b(\mathrm{Coh}^G(T^*(G/P)))$ arising from the noncommutative Springer resolution, as in \cite{AB, BezTilt, BM}. We explain how this collection gives rise to an alternative construction of a lift of the permutation representation $\mathsf{k}[(G/P)(\mathbb{F}_q)]$ from $G(\mathbb{F}_q)$ to the algebraic group $G$. We then show that this lift agrees with the main construction of the present paper, thereby giving a natural geometric interpretation for our construction.

The idea to compare these constructions comes from the following observation.
\begin{lemma}\label{lem:diagonal}
    Let $X$ be a projective variety over $\mathbb{F}_q$ with an action of the algebraic group $G$, then we can consider the representation $\mathsf{k}[X(\mathbb{F}_q)]$ of $G(\mathbb{F}_q)$. Suppose we have a decomposition
    \begin{equation}\label{eqn:odelta}
        [\mathcal{O}_\Delta] = \sum_i [\mathcal{F}_i\boxtimes \mathcal{F}_i']
    \end{equation}
    in $K_0(\mathrm{Coh}^G(X^2))$. Then the virtual $\mathsf{k}$-representation of $G$ given by
    \begin{equation}\label{eqn:virtualcoh}
        \sum_i R\Gamma(\mathcal{F}_i \otimes \mathrm{Fr}^*(\mathcal{F}_i'))
    \end{equation}
    is sent to $\mathsf{k}[X(\mathbb{F}_q)]$ under restriction from $G$ to $G(\mathbb{F}_q)$.
\end{lemma}
\begin{proof}
    We let $\Delta \colon  X \to X\times X$ be the diagonal morphism, and $g \colon  X \to X\times X$ be the morphism given by $x \mapsto (x, \mathrm{Fr}(x))$. Obviously, $\Delta$ is equivariant for the diagonal action of $G$ on $X^2$, while $g$  is equivariant for the action of $G$ on $X^2$
    coming from the standard action on the first factor and the standard action twisted by Frobenius on the second one. In particular, both maps
    are equivariant for the diagonal action of $G(\mathbb{F}_q)$.
    
    We then apply the pullback functor $g^*$ to objects in the two sides of \eqref{eqn:odelta}. On the right-hand side, 
    the pullback of the $G^2$ equivariant object $\mathcal{F}_i\boxtimes \mathcal{F}_i'$ is the $G$-equivariant object
    $\mathcal{F}_i \otimes \mathrm{Fr}^*(\mathcal{F}_i').$
    On the left-hand side, we get $g^*\mathcal{O}_\Delta =\mathcal{O}_{\Delta(X)\cap g(X)}$, 
    where the intersection $\Delta(X)\cap g(X)$ is the disjoint union of points in $X(\mathbb{F}_q)$.

    In view of the above equivariance compatibility, we get an equality
    $$[\mathcal{O}_{X({\mathbb F}_q)} ] =\sum [\mathcal{F}_i\otimes \mathrm{Fr}^*(\mathcal{F}_i')]$$
    in the $G(\mathbb{F}_q)$-equivariant Grothendieck group of $X$; applying the functor of global sections to both sides we get the result. 
\end{proof}

\subsection{Comparison of bases and representations}


We begin in the more standard case of $G/B$ for simplicity, and will then in \S \ref{sec:categorification-gp} explain how to generalize the arguments presented in this subsection to the more general case of $G/P$. The objects we will introduce in this section arise in the study of the noncommutative Springer resolution and the constructible-coherent correspondence as in \cite{AB}; this correspondence and these objects are explained in \cite{BezTilt} and \cite{BM}. The results in loc.\ cit.\ establish a derived equivalence
\begin{align}
    D^b(\mathrm{Coh}^G(T^*(G/B))) \cong D^b(\mathsf{A}\mathrm{-mod}^G)\label{eqn:derivedequiv}
\end{align}
for an algebra $\mathsf{A}$ with a $G$-action. The categories in \eqref{eqn:derivedequiv} can then also be identified with the derived category of Iwahori-Whittaker sheaves on the affine flag variety, which is the main result of \cite{AB}. The image of the tautological $t$-structure on the right-hand side of (\ref{eqn:derivedequiv}) in $D^b(\mathrm{Coh}^G(T^*(G/B)))$ is called the noncommutative Springer $t$-structure.

We note that the bases considered in this section are closely related to the Kazhdan-Lusztig basis for the affine Hecke algebra and its spherical quotient, along with the canonical basis for Lusztig's periodic Hecke module as introduced in \cite{LPeriodic1, LPeriodic2}. Thus we will make use of the extended affine Weyl group $W^{\mathrm{aff}}$ and its Kazhdan-Lusztig basis $\{C_w'\}_{w \in W^{\mathrm{aff}}}$. We will also use the natural inclusion $t : X^*(T) \to W^{\mathrm{aff}}$.

For the remainder of this section, we assume that $G$ is simply connected, which guarantees the $G$-equivariance in the following definition.
\begin{defn}
    Let $\{E_i\}$ (resp.\ $\{E^i\}$) denote the collection of irreducible (resp. projective) modules over the algebra $\mathsf A$. These modules are constructed explicitly in \cite{BM} and each is shown to carry a $G$-equivariant structure. Thus we let $\{\mathcal{E}_i\}$ and $\{\mathcal{E}^i\}$ denote their images in $D^b(\mathrm{Coh}^G(T^*(G/B)))$ under the equivalence in (\ref{eqn:derivedequiv}).
\end{defn}

\begin{defn}
    For any $w \in W$, let 
    \begin{align}
        g_w & = \frac{1}{|W_{I(w)}|}C_{w \cdot t(\varpi_{\overline{I(w)}})}' \cdot 0 \in N,
    \end{align}
    where $C_{y}'$ is the affine Kazhdan-Lusztig basis element corresponding to $y \in W^{\mathrm{aff}}$.
\end{defn}

The next proposition will use the natural identification $K_0(\mathrm{Coh}^G(T^*(G/B))) \cong K_0(\mathrm{Coh}^G(G/B)) \cong \mathbb{Z}[T]$. We continue to use the natural action of $\mathbb{Q}[W]$ on $\mathbb{Q}[T]$.
\begin{prop}\label{prop:ewkl}
    The collections $\{\mathcal{E}_i\}$ and $\{\mathcal{E}^i\}$ have a natural indexing by the Weyl group, and thus can be written $\{\mathcal{E}_w\}_{w \in W}$, $\{\mathcal{E}^w\}_{w \in W}$. 
    
    Further, under the isomorphism
    \[
    K_0(\mathrm{Coh}^G(T^*(G/B))) \cong N,
    \]
    the image of $[\mathcal{E}_w]$ is $g_w$, while the image of $[\mathcal{E}^w]$ is the basis vector dual to $g_w$ with respect to the Euler pairing.
\end{prop}
\begin{proof}

    First, note that there is a natural map 
    \begin{align}\label{eqn:forgetgt}
        D^b(\mathrm{Coh}^G(T^*(G/B))) \to D^b(\mathrm{Coh}^T(T^*(G/B)))
    \end{align}
    obtained by forgetting $G$-equivariance. For $\mathcal{F} \in D^b(\mathrm{Coh}^G(T^*(G/B)))$, we denote by $\mathcal{F}^T$ its image in $D^b(\mathrm{Coh}^T(T^*(G/B)))$ under this map.

    Recall that $K_0(\mathrm{Coh}^T(G/B))$ can naturally be identified with Lusztig's periodic Hecke module as defined in \cite{LPeriodic1, LPeriodic2}. As explained in loc.\ cit. and \cite{Lk2}, the periodic Hecke module has a canonical basis $\{\tilde{C}_w'\}_{w \in W^{\mathrm{aff}}}$.

    First we claim that elements of the form $[\mathcal{E}_i^T(\lambda)] \in K_0(\mathrm{Coh}^T(T^*(G/B)))$ across all $i$ and across all $\lambda \in X^*(T)$ are exactly canonical basis elements of the form $\tilde{C}_w'$ for $w \in W^{\mathrm{aff}}$. By their construction the objects $\mathcal{E}_i^T$ are each set-theoretically supported on the zero section, and since $\mathrm{pr}_* : K_0(\mathrm{Coh}^T(T^*(G/B))) \to K_0(\mathrm{Coh}^T(G/B))$ is an isomorphism when restricted to such objects, it suffices to show this for the objects $[\mathrm{pr}_*(\mathcal{E}_i^T(\lambda))] \in K_0(\mathrm{Coh}^T(G/B))$. This is exactly what is shown in \cite[\S 5.3.6]{BM} (for $e = 0$). The main result of \cite{Lk2} ensures that the canonical bases appearing in this theorem (which are defined from a $K$-theoretic perspective in \cite{Lk1}) agree with the basis $\tilde{C}_w'$ defined from a combinatorial perspective in \cite{LPeriodic1, LPeriodic2}. Similarly, each $[(\mathcal{E}^i)^T(\lambda)]$ is a dual basis element under the natural pairing on the periodic Hecke module defined combinatorially in \cite{LPeriodic2} and $K$-theoretically in \cite{Lk1}, c.f.\ \cite[\S 5.3.6]{BM}. 

    Since this shows that for every $i$, the elements $[\mathcal{E}_i^T(\lambda)]$ with $\lambda \in X^*(T)$ form a full $X^*(T)$-orbit of basis elements, we obtain a bijection from the set of indices $i$ to the set $W^{\mathrm{aff}}/X^*(T) \cong W$. Thus for any $w \in W$, we have characterized $\mathcal{E}_w^T$ as the unique object in the image of the forgetful functor \eqref{eqn:forgetgt} such that $\{[\mathcal{E}_w^T(\lambda)]\}_{\lambda \in X^*(T)} = \{\tilde{C}_{w \cdot t(\mu)}'\}_{\mu \in X^*(T)}$.

    Further, it follows from the identifications of \cite{BM, Lk2} that $[\mathcal{E}_w^T] = \frac{1}{|\mathrm{stab}_W(\mu)|}\tilde{C}_{w \cdot t(\mu)}'$ for the unique $\mu \in X^*(T)$ satisfying $\langle \mu, \alpha_i^\vee\rangle = 0$ if $w(\alpha_i)$ is a positive root and $\langle \mu, \alpha_i^\vee\rangle = 1$ if $w(\alpha_i)$ is a negative root. We note that $\mu = \varpi_{\overline{I(w)}}$ satisfies these conditions, and thus for any $w \in W$ we have
    \[
    [\mathcal{E}_w^T] = \frac{1}{|W_{I(w)}|}\tilde{C}_{w \cdot t(\varpi_{\overline{I(w)}})}'.
    \]

    Under the identification between $\mathbb{C}[T]$ and $K_0(\mathrm{Coh}^G(T^*(G/B)))$, and between the periodic module and $K_0(\mathrm{Coh}^T(T^*(G/B)))$, the forgetful map in (\ref{eqn:forgetgt}) sends $C_{w \cdot t(\varpi_{\overline{I(w)}})}' \cdot 0 \in \mathbb{C}[T]$ to $\tilde{C}_{w \cdot t(\varpi_{\overline{I(w)}})}'$ (see equation (54) and the subsequent remark in \cite[\S 6.1.8]{ABBGM}). 

    The statement about $[\mathcal{E}^w]$ follows from the same discussion together with the compatibility of the natural Euler pairings under the above identifications.
\end{proof}
For any smooth $G$-variety $X$ and $\mathcal{F} \in D^b(\mathrm{Coh}^G(X))$, we let $\mathcal{F}^\vee$ be its image under the Grothendieck duality functor $\mathcal{R}Hom_X(-, \mathcal{O}_{X})$. In particular, for any $\mathcal{F}, \mathcal{G} \in D^b(\mathrm{Coh}^G(G/B))$, we have
\[
\mathrm{RHom}_{G/B}(\mathcal{F}, \mathcal{G}) = R\Gamma(G/B, \mathcal{F}^\vee\otimes \mathcal{G}),
\]
where the tensor product is left-derived.

In the following, we let $i : G/B \to T^*(G/B)$ be the inclusion of the zero section, and $\mathrm{pr} : T^*(G/B) \to G/B$ be the projection. 
\begin{defn}\label{def:ve}
    We define a virtual $\mathsf{k}$-representation of the algebraic group $G$ by
    \begin{align}
        \tilde{V}_\mathcal{E} & = \sum_{w} \mathrm{RHom}(i^*\mathcal{E}^{w}, \mathrm{Fr}^*(\mathrm{pr}_*\mathcal{E}_{w})).
    \end{align}
\end{defn}

We now state the main theorem of this section.
\begin{theorem}\label{thm:twolifts}
    Let $V = \mathsf{k}[(G/B)(\mathbb{F}_q)]$. Then the virtual $\mathsf{k}$-representations $\tilde{V}$ (constructed in Section \ref{sec:brauer} using the basis $\{f_w\}_{w \in W}$) and $\tilde{V}_{\mathcal{E}}$ (constructed in~Definition~\ref{def:ve}, see also Proposition \ref{prop:exceptionallift}) are equal.

    Moreover, if $w \in \mathcal{J}$, $d$ is the Duflo involution in the same left cell as $w$, and
    \begin{align}
        \mathcal{M}_w & = \mathrm{RHom}(i^*\mathcal{E}^{w_0w}, \mathrm{Fr}^*(\mathrm{pr}_*\mathcal{E}_{w_0d})),
    \end{align}
    then $M_w = [\mathcal{M}_w]$ in the Grothendieck group of virtual $\mathsf{k}$-representations of $G$.
\end{theorem}

To prove Theorem \ref{thm:twolifts}, we now introduce an intermediate object $\tilde{V}_{\mathrm{KL}}$, which we will then compare to both $\tilde{V}$ and $\tilde{V}_{\mathcal{E}}$. Recall that in Section \ref{sec:brauer}, we explained that while the lift $\tilde{V}$ of $V$ to a representation of the algebraic group $G$ over $\mathsf{k}$ we constructed used the specific choice of basis $f_w$ for the $A$-module $N$, such a lift is still well-defined for any $A$-basis for $N$ satisfying the axioms in Lemma \ref{lem:fil}. As such, we can make the following definition.
\begin{defn}
    Let $\tilde{V}^{\mathrm{KL}}$ be the lift of $V$ arising by the construction of Section \ref{sec:brauer} from the basis $\{g_w\}_{w \in W}$. This collection is a basis for $N$ as an $A$-module.
\end{defn}

To prove Theorem \ref{thm:twolifts}, we first show that $\tilde{V}_{\mathcal{E}} \cong \tilde{V}_{\mathrm{KL}}$, then $\tilde{V}_{\mathrm{KL}} \cong \tilde{V}$. 

\begin{prop}\label{prop:klande}
    The virtual representations $\tilde{V}^{\mathrm{KL}}$ and $\tilde{V}_{\mathcal{E}}$ are equal.
\end{prop}
\begin{proof}
We have
    \begin{equation}
    \tilde{V}_{\mathcal{E}} = \sum_{w \in W} R\Gamma_{G/B}(i^*(\mathcal{E}^{w})^\vee\otimes \mathrm{Fr}^*(\mathrm{pr}_*\mathcal{E}_{w})).\label{eqn:kllift}
    \end{equation}
    By Proposition \ref{prop:ewkl}, the classes of $\{\mathcal{E}_w\}_{w \in W}$ identify with the basis $\{g_w\}_{w \in W}$, while the classes of $\{\mathcal{E}^w\}_{w \in W}$ identify with the dual basis. Thus \eqref{eqn:kllift} is exactly the expression for the lift $\tilde{V}^{\mathrm{KL}}$ arising from the basis $\{g_w\}_{w \in W}$.
\end{proof}

\begin{prop}\label{prop:exceptionallift}
    The restriction of $\tilde{V}_{\mathcal{E}}$ to $G(\mathbb{F}_q)$ is the representation $\mathsf{k}[(G/B)(\mathbb{F}_q)]$.
\end{prop}
\begin{proof}
    From \eqref{eqn:kllift}, we have
    \begin{equation}
    \tilde{V}_{\mathcal{E}} = \sum_{w \in W}R\Gamma_{G/B}(i^*(\mathcal{E}^{w})^\vee \otimes \mathrm{Fr}^*(\mathrm{pr}_*\mathcal{E}_{w})).
    \end{equation}
    Thus it remains to show that
    \begin{align}
        [\mathcal{O}_{\Delta}] = \sum_w [(\mathcal{E}^w)^\vee \boxtimes \mathcal{E}_{w}]
    \end{align}
    in $K_0(\mathrm{Coh}^G(T^*(G/B)^2))$. By the equivalence to $D^b(\mathsf A\mathrm{-mod}^G)$ established in \cite{AB, BM}, this reduces to checking
    \begin{align}
        [\mathsf A_{\mathrm{diag}}] = \sum_w [(E^w)^\vee \boxtimes E_w],
    \end{align}
    where $\mathsf A_{\mathrm{diag}}$ is the diagonal $\mathsf A$-bimodule. This follows from the fact that $\{E^w\}_{w \in W}$ are the indecomposable projective $\mathsf A$-modules and $\{E_w\}_{w \in W}$ are the corresponding irreducible modules.
\end{proof}

Recall that by the main construction of the present paper, via the basis $f_w$ and its dual we construct a lift of any unipotent representation $\rho$ of $G(\mathbb{F}_q)$ over $\mathsf{k}$ to a virtual representation $\tilde{\rho}$ of the algebraic group $G$ which can be explicitly expressed as a linear combination of Weyl modules. In turn, Proposition \ref{prop:exceptionallift} gives an alternative way to construct such lifts given the input of the objects $\mathcal{E}$. We now explain the main result of this section: showing that $\tilde{V}_{\mathcal{E}}$ is a lift which agrees with the main construction of the present paper.

\begin{remark}
We initially investigated the possibility that applying Proposition~\ref{prop:exceptionallift} to the exceptional collection appearing in \cite{SVK}, rather than to $\mathcal{E}$, would provide a lift of the representation $\mathsf{k}[(G/B)(\mathbb{F}_q)]$ which agrees with the main construction of the present paper. Using our Sage code \cite{github} along with Mathematica code which was generously shared with us by the authors of \cite{SVK}, we observed that these constructions do not agree; we found a counterexample in Type $A_4$ (after checking that the two constructions do agree in Types $A_2, B_2, G_2, A_3, B_3$).
\end{remark}

The next subsection is devoted to the proof of Theorem \ref{thm:twolifts}.

\subsection{Proof of Theorem \ref{thm:twolifts}}

Since Proposition \ref{prop:klande} establishes that $\tilde{V}^{\mathrm{KL}}$ and $\tilde{V}_{\mathcal{E}}$ are equal, we will conclude in~Proposition \ref{prop:klandv} by arguing that $\tilde{V}^{\mathrm{KL}}$ and $\tilde{V}$ are equal, completing the proof of~Theorem \ref{thm:twolifts}. To do so, we will need to compare the bases $\{f_w\}_{w \in W}$ and $\{g_w\}_{w \in W}$ modulo lower two-sided Kazhdan-Lusztig cells; to do this, we need to establish a few combinatorial results.

\subsubsection{Combinatorial comparison of $f_w$ and $g_w$}
For any two-sided cell $c \subset W$, let $c^a \subset W^{\mathrm{aff}}$ denote the set of all elements of the form $w\cdot y$ where $w \in c \subset W$ and $\ell(wy) = \ell(w) + \ell(y)$. Equivalently, $c^a$ is the set of all elements of the form $w \cdot y$ where $w \in c$, $y \in {}^fW$, and $\ell(wy) = \ell(w) + \ell(y)$, where ${}^fW$ is the set of elements in $W^{\mathrm{aff}}$ which are minimal in their left $W$-coset. We use the same notation $c_L^a$ when $c_L$ is a left cell for $W$.
\begin{lemma}\label{lem:twosided}
For a two-sided cell $c$ for $W$, write $\mathbb{Z}[W]_{\leq c}^a = \mathbb{Z}[W]_{\leq c} \cdot \mathbb{Z}[W^{\mathrm{aff}}]$. Then
\begin{align}
    \mathbb{Z}[W]_{\leq c}^a = \mathrm{span}\{C_x'~|~x \in (c')^a \text{ for some two-sided cell } c' \leq c\}.
\end{align}
\end{lemma}
\begin{proof}
    We must show that $x \in (c')^a$ for some two-sided cell $c' \leq c$ if and only if
    $C_x' \in \mathbb{Z}[W]_{\leq c}^a$. First suppose that
    $C_x' \in \mathbb{Z}[W]_{\leq c}^a$. The left $\mathbb{Z}[W]$-module
    $\mathbb{Z}[W^{\mathrm{aff}}]$ has a filtration indexed by ${}^fW$ given by
    \[
    \mathbb{Z}[W^{\mathrm{aff}}]_{\leq y} = \langle \mathbb{Z}[W]y' \rangle_{y' \leq y}
    \]
    for $y', y \in {}^fW$. Write $x = w \cdot y$ for $w \in W$ and
    $y \in {}^fW$. The assumption that $C_x' \in \mathbb{Z}[W]_{\leq c}^a$
    implies that $C_x' = b_1b_2$ for $b_1 \in \mathbb{Z}[W]_{\leq c}$ and
    $b_2 \in \mathrm{span}\{y'\}_{y' \in {}^fW}$. So projecting to the
    associated graded, we see that we must have $b_1 = C_w'$, which implies that
    \[
    C_x' = \sum_{y' \in {}^fW} r_{y'} C_w' \cdot y'
    \]
    for some coefficients $r_{y'}$. When expressing the right-hand side in terms
    of the affine Kazhdan-Lusztig basis, its leading term (with respect to the
    Bruhat order) must be of the form $C_{w \cdot y}'$, which we can see must be
    $C_x'$ by comparing to the left-hand side. This means $x = w \cdot y$ for
    some $w \in c'$ and $y \in {}^fW$, where $c' \leq c$, which then gives
    $x \in (c')^a$.

    The same associated graded argument can be used to show that
    $\mathbb{Z}[W]_{\leq c}^a$ is a based submodule of
    $\mathbb{Z}[W^{\mathrm{aff}}]$ with respect to the Kazhdan-Lusztig basis.
    Once this is established, the converse is then clear since if
    $x \in (c')^a$ for some $c' \leq c$, then writing $x = w \cdot y$ we
    clearly have $C_w'C_y' \in \mathbb{Z}[W]_{\leq c}^a$, which by the based
    property then shows that $C_{w \cdot y}' = C_x' \in \mathbb{Z}[W]_{\leq c}^a$.
\end{proof}

Let $\mathbb{Z}[W]_{c}^a = \mathbb{Z}[W]_{\leq c}^a/\mathbb{Z}[W]_{< c}^a$. For $w \in c^a$ let $\overline{C}_w'$ be the image of the canonical basis in $\mathbb{Z}[W]_{c}^a$; one then observes the following.

\begin{lemma}\label{lem:leftcell}
    We have
    \begin{align}
        \mathbb{Z}[W]_{c}^a \cong \bigoplus_{c_L \subset c} \mathbb{Z}[W]_{c_L}^a
    \end{align}
    where each $\mathbb{Z}[W]_{c_L}^a$ is the span of $\overline{C}_w'$ for $w \in c_L^a$, and the direct sum is taken over all left cells $c_L$ contained in the two-sided cell $c$.
\end{lemma}

Let $z_{\lambda} \in \mathbb{Z}[T]^W$ be the element of the invariant ring corresponding to any dominant weight $\lambda$. We can do the same decompositions as above for the module $N$ after acting on the weight $0$ by elements of $\mathbb{Z}[W]$ (this corresponds to identifying any two elements which differ only by multiplication by an element of $W$ on the right, i.e.\ the usual identification between the spherical quotient of $\mathbb{Z}[W^{\mathrm{aff}}]$ and the group algebra of weights).

We note that $N$ is a free module over $A$ with basis given by the images $C_z' \cdot 0$ for $z$ ranging across the set of weights $\{w \cdot t(\varpi_{\overline{I(w)}})\}_{w \in W}$. Combining this with the left cell decomposition above, we get the following.
\begin{lemma}\label{lem:thetas}
    For any left cell $c_L \subset W$, $N_{c_L}$ is freely generated over $A$ by the images of the elements $g_y$ for which $y \cdot t(\varpi_{\overline{I(y)}}) \in c_L^a$. 
\end{lemma}

\subsection{Comparison of bases up to lower two-sided cells}
For this section, write $\lambda_1 \leq \lambda_2$ if and only if $\lambda_1$ lies in the convex hull of $W\cdot \lambda_2$. Let $N^{\leq \lambda}$ denote the $\mathbb{C}$-subspace of elements $m \in N$ such that for every $\mu \in \mathrm{supp}(m)$, $\mu \leq \lambda$ in this order; we define $A^{\leq \lambda}$ similarly, and write $N^{<\lambda} = \sum_{\mu < \lambda} N^{\leq \mu}$. Note that with these definitions, $N$ is a filtered module over the filtered algebra $A$, and we can consider the associated graded module $\mathrm{gr}(N)$ over $\mathrm{gr}(A)$. We now claim that the basis $\{f_w\}_{w \in W}$ is compatible with this filtration as follows.

\begin{lemma}\label{lem:leqspan}
    For any $w \in W$,
    \begin{align}
        N^{\leq \varpi_{\overline{I(w)}}} \subset \mathrm{span}_A\{f_y~|~y \in W, I(y) \supseteq I(w)\}.
    \end{align}
    In particular, 
    \begin{align}\label{eqn:grspan}
        \mathrm{gr}_{\varpi_{\overline{I(w)}}}(N) = \mathrm{span}_{\mathbb{C}} \{z_{\varpi_{\overline{I(w)}}-\varpi_{\overline{I(y)}}}f_y~|~I(y) \supseteq I(w)\}.
    \end{align}
\end{lemma}
\begin{proof}
    We first prove the second claim. If $I(y) \supseteq I(w)$, then $\varpi_{\overline{I(w)}}-\varpi_{\overline{I(y)}}$ is dominant. By definition, every weight occurring in $f_y$ lies in the $W$-orbit of $\varpi_{\overline{I(y)}}$, so $f_y \in N^{\leq \varpi_{\overline{I(y)}}}$. Since $z_{\varpi_{\overline{I(w)}}-\varpi_{\overline{I(y)}}} \in A^{\leq \varpi_{\overline{I(w)}}-\varpi_{\overline{I(y)}}}$, it follows that
    \[
    z_{\varpi_{\overline{I(w)}}-\varpi_{\overline{I(y)}}}f_y \in N^{\leq \varpi_{\overline{I(w)}}}.
    \]
    Moreover, its image in $\mathrm{gr}_{\varpi_{\overline{I(w)}}}(N)$ is nonzero, since $f_y$ contains the weight $\exp(\varpi_{\overline{I(y)}})$ with nonzero coefficient.

    Next note that
    \begin{align}
        |\{y \in W~|~I(y) \supseteq I(w)\}| = |W/W_{\overline{I(w)}}| = \dim_{\mathbb{C}} \mathrm{gr}_{\varpi_{\overline{I(w)}}}(N).
    \end{align}
    The elements appearing on the right-hand side of \eqref{eqn:grspan} are $\mathbb{C}$-linearly independent because the basis $\{f_y\}_{y \in W}$ is $A$-linearly independent and the coefficients $z_{\varpi_{\overline{I(w)}}-\varpi_{\overline{I(y)}}}$ are nonzero elements of the domain $A$. Thus the images of these elements in $\mathrm{gr}_{\varpi_{\overline{I(w)}}}(N)$ form a linearly independent subset of full dimension, proving the second claim.

    We now prove the first claim by descending induction on $|\overline{I(w)}|$. Let $m \in N^{\leq \varpi_{\overline{I(w)}}}$. By the second claim, the image of $m$ in $\mathrm{gr}_{\varpi_{\overline{I(w)}}}(N)$ can be written as the image of an element in $\mathrm{span}_A\{f_y~|~I(y) \supseteq I(w)\}$. Subtracting such an element from $m$, we may assume that $m \in N^{< \varpi_{\overline{I(w)}}}$. The induction hypothesis applied to the smaller dominant weights occurring in this filtration then implies that $m$ itself lies in $\mathrm{span}_A\{f_y~|~I(y) \supseteq I(w)\}$.
\end{proof}
In particular, note that by the definition of $f_w$ in terms of the Kazhdan-Lusztig basis, Lemma \ref{lem:leqspan} implies that $\mathrm{gr}_{\varpi_{\overline{I(w)}}}(N)$ breaks into a direct sum of left cell modules corresponding to those left cells containing elements whose right descent sets contain $I(w)$. 
\begin{lemma}\label{lem:maintwosided}
    The images of the elements $f_w$ and $g_w$ agree in $N_c$, where $c$ is the two-sided cell containing $w$.
\end{lemma}
\begin{proof}
    First note that by Lemma \ref{lem:twosided}, both $f_w$ and $g_w$ lie in $N_{\leq c}$. Note that $N_c$ splits into a direct sum $\oplus_{c_L \subset c} N_{c_L}$ of left cell modules over left cells $c_L$ contained in the two-sided cells $c$. By Lemma \ref{lem:leftcell}, they must lie in the same left cell submodule $N_{c_L}$. Further, note that it is clear from the definitions of $f_w$ and $g_w$ that they agree after restriction to the weight filtration subquotient $\mathrm{gr}_{\varpi_{\overline{I(w)}}}(N)$; this follows from the fact that affine Kazhdan-Lusztig polynomials coincide with ordinary Kazhdan-Lusztig polynomials when restricted to a given orbit under the left action of the finite Weyl group. 

    This means that $g_w - f_w$ lies in $N^{<\varpi_{\overline{I(w)}}}$. Thus by Lemma \ref{lem:leqspan}, we can write
    \begin{align}
        g_w - f_w & = \sum_{I(y) \supsetneq I(w)} a_yf_y
    \end{align}
    for some $a_y \in A$. However, each $f_y$ on the right-hand side lies in a left cell submodule distinct from $c_L$, since right descent sets (and thus the sets $I(y)$) are constant within left cells. Thus the image of the right-hand side in the cell subquotient $N_{c_L}$ is trivial. But since the image of the left-hand side in $N_c \cong \oplus_{c_L \subset c} N_{c_L}$ lies in the summand $N_{c_L}$, this means the image of the right-hand side in $N_{c}$ must also be trivial, completing the proof. 
\end{proof}
\subsubsection{Completing the proof of Theorem \ref{thm:twolifts}}
\begin{prop}\label{prop:klandv}
    The representations $\tilde{V}^{\mathrm{KL}}$ and $\tilde{V}$ are equal. 
\end{prop}
\begin{proof}
    Note that the results of Section \ref{sec:klstein} guarantee that $\tilde{V}$ depends only on the image of the basis $\{f_w\}_{w \in W}$ in the associated graded arising from the two-sided cell filtration. In Lemma \ref{lem:maintwosided}, we showed that this image is exactly the same as the image of the basis $\{g_w\}_{w \in W}$ in this associated graded; thus the two representations $\tilde{V}^{\mathrm{KL}}$ and $\tilde{V}$ agree.
\end{proof}

By Propositions \ref{prop:klande} and \ref{prop:klandv},
\begin{align}
    \tilde{V}_{\mathcal{E}} = \tilde{V}^{\mathrm{KL}} = \tilde{V},
\end{align}
so it remains only to prove the final assertion of Theorem \ref{thm:twolifts}. Let $c'$ be the two-sided cell containing $w_0w$ (equivalently, containing $w_0d$). By Proposition \ref{prop:ewkl}, the class of $\mathcal{E}_{y}$ is $g_y$, while the class of $\mathcal{E}^{y}$ is the dual basis vector. Hence the class $[\mathcal{M}_w]$ is the coefficient of $g_{w_0w}$ in $\tilde{\phi}(g_{w_0d})$. Because $\tilde{\phi}$ preserves the two-sided cell filtration, this is also the coefficient of $\bar g_{w_0w}$ in the image of $\bar g_{w_0d}$ under the induced endomorphism of $N_{c'}$. By Lemma \ref{lem:maintwosided}, the images of $g_y$ and $f_y$ in $N_{c'}$ agree for every $y \in c'$, so this coefficient is also the coefficient of $\bar f_{w_0w}$ in the image of $\bar f_{w_0d}$. By the proof of Proposition \ref{prop:jring}, this coefficient is exactly $M_w$. Thus $[\mathcal{M}_w] = M_w$, completing the proof of Theorem \ref{thm:twolifts}.

\subsection{The case of $G/P$}\label{sec:categorification-gp}

We now explain how to deduce the corresponding statement for a general
parabolic quotient $G/P$ from the case of $G/B$ already treated above.
We continue to assume that $G$ is simply connected.

Let $P\supset B$ be a parabolic subgroup, let $L\subset P$ be the Levi subgroup
containing $T$, and let $W_P\subset W$ be the corresponding parabolic subgroup of $W$ such that $W_P$ is the Weyl group of $L$; note that $W_B = \{1\}$ and $W_G = W$.
Write
\[
q_P\colon G/B\to G/P
\]
for the projection.
We also write
\[
i_P\colon G/P\hookrightarrow T^*(G/P), \qquad
\mathrm{pr}_P\colon T^*(G/P)\to G/P
\]
for the zero section and bundle projection respectively.

The projection $q_P$ gives rise to the correspondence
\[
T^*(G/P)\xleftarrow{\ \mathrm{pr}_1\ } T^*(G/P)\times_{G/P}G/B \xrightarrow{\ i\ } T^*(G/B),
\]
and hence to functors
\[
q_P^\star(\mathcal F)=i_*\mathrm{pr}_1^*(\mathcal F)\otimes \O(\rho),
\qquad
(q_P)_\star(\mathcal G)=(\mathrm{pr}_1)_*i^*(\mathcal G(-\rho)).
\]
We will use the fact proved in \cite{BM} that, for the exotic $t$-structures,
$q_P^\star$ is conservative and $t$-exact, and that just as
$\mathcal E := \bigoplus_w \mathcal E^w$ is a projective generator for the exotic heart on
$T^*(G/B)$,
\[
\mathcal E_{P} = (q_P)_\star(\mathcal E)
\]
is a projective generator for the corresponding parabolic exotic heart on
$T^*(G/P)$.

\begin{lemma}\label{lem:k0gp}
There are natural identifications
\[
K_0(\mathrm{Coh}^G(T^*(G/P)))
\cong
K_0(\mathrm{Coh}^G(G/P))
\cong
\mathbb{Z}[T]^{W_P}.
\]
Under these identifications, the pullback
\[
q_P^*\colon K_0(\mathrm{Coh}^G(G/P))\to K_0(\mathrm{Coh}^G(G/B))
\]
is the natural inclusion
\[
\mathbb{Z}[T]^{W_P}\hookrightarrow \mathbb{Z}[T].
\]
\end{lemma}

Let $N_P=\mathbb{C}[T]^{W_P}$, an $A = \mathbb{C}[T]^W$-module which we naturally consider as a submodule of $N$. Let
\[
\mathrm{Av}_P\colon N\to N_P
\]
be the averaging operator
\[
\mathrm{Av}_P(f)=\frac{1}{|W_P|}\sum_{w\in W_P}w(f).
\]
Let $W^P\subset W$ be the set of minimal length representatives for $W/W_P$.

Let
\[
\mathsf{A}_P = \mathrm{End}(\mathcal E_{P})^{op}.
\]
Under the derived equivalence
\[
D^b(\mathrm{Coh}^G(T^*(G/P))) \cong D^b(\mathsf{A}_P\mathrm{-mod}^G),
\]
let
\[
\{\mathcal E_{P,\alpha}\}_{\alpha\in W^P},
\qquad
\{\mathcal E_P^{\alpha}\}_{\alpha\in W^P}
\]
denote the collections corresponding to irreducible and projective
$\mathsf A_P$-modules respectively.

\begin{defn}\label{def:vep}
We define a virtual $\mathsf{k}$-representation of the algebraic group $G$ by
\[
\tilde V_{\mathcal E,P}
=
\sum_{\alpha\in W^P}
\mathrm{RHom}\!\left(i_P^*\mathcal E_P^{\alpha},
\mathrm{Fr}^*(\mathrm{pr}_{P*}\mathcal E_{P,\alpha})\right).
\]
\end{defn}

\begin{prop}\label{prop:gp-categorical-lift}
The restriction of $\tilde V_{\mathcal E,P}$ to $G(\mathbb{F}_q)$ is the permutation
representation
\[
\mathsf{k}[(G/P)(\mathbb{F}_q)].
\]
\end{prop}

\begin{proof}
Exactly as in Proposition \ref{prop:exceptionallift}, it suffices to show that
\[
[\mathcal O_\Delta]
=
\sum_{\alpha\in W^P}
[(\mathcal E_P^{\alpha})^\vee\boxtimes \mathcal E_{P,\alpha}]
\]
in
\[
K_0(\mathrm{Coh}^G(T^*(G/P)\times T^*(G/P))).
\]
Under the equivalence with $\mathsf{A}_P$-modules, this reduces to the corresponding statement for the diagonal $\mathsf A_P$-bimodule, which follows from the fact that $\{\mathcal E_P^\alpha\}_{\alpha \in W^P}$ and $\{\mathcal E_{P,\alpha}\}_{\alpha \in W^P}$ correspond to the indecomposable projective and irreducible $\mathsf A_P$-modules.
\end{proof}

We next explain how to compare this categorical lift with the lift obtained from
the explicit basis construction of Section \ref{sec:brauer}.

In the $G/B$ case, we compared two bases of $N$ over $A$:
the basis $\{g_w\}_{w\in W}$ coming from the objects $\mathcal E_w$,
and the basis $\{f_w\}_{w\in W}$ coming from the main construction.
We set
\[
g_w^P = \mathrm{Av}_P(g_w), \qquad
f_w^P = \mathrm{Av}_P(f_w), \qquad w\in W^P.
\]

\begin{prop}\label{prop:parabolic-descent}
The collections $\{g_w^P\}_{w\in W^P}$ and $\{f_w^P\}_{w\in W^P}$
are $A$-bases of $N_P$.

Moreover, under
\[
K_0(\mathrm{Coh}^G(T^*(G/P)))\cong N_P,
\]
the classes of the irreducible objects
$\{\mathcal E_{P,\alpha}\}_{\alpha\in W^P}$ identify with the basis
$\{g_w^P\}_{w\in W^P}$, while the classes of the projective objects
$\{\mathcal E_P^{\alpha}\}_{\alpha\in W^P}$ identify with the dual basis. The basis entering the construction of the lift from
Section \ref{sec:brauer} is $\{f_w^P\}_{w\in W^P}$.

Finally, if one equips $N_P$ with the filtration induced from the two-sided
cell filtration on $N$, then $g_w^P$ and $f_w^P$ have the same image in the
associated graded.
\end{prop}

\begin{proof}
On the categorical side, it follows from the description of the parabolic projective
generator as $(q_P)_\star(\mathcal E)$ and from
Proposition \ref{prop:ewkl} that the corresponding classes in $K_0$ are obtained
from the $G/B$ classes by passage from $N$ to the parabolic module $N_P$.

On the combinatorial side, the construction of Section \ref{sec:brauer} is functorial for
the finite map $T/W_P\to T/W$, and hence gives the descended basis
$\{f_w^P\}_{w\in W^P}$ of $N_P$.

Since the operator $\mathrm{Av}_P$ is $A$-linear and preserves the two-sided cell
filtration, the equality of associated-graded images proved in the $G/B$ case
(Lemma \ref{lem:maintwosided}) descends
immediately to the parabolic module.
\end{proof}

Let $\tilde V_P^{\mathrm{KL}}$ be the lift of $\mathsf{k}[(G/P)(\mathbb{F}_q)]$
constructed from the basis $\{g_w^P\}_{w\in W^P}$, and let $\tilde V_P$ be the
lift constructed from the basis $\{f_w^P\}_{w\in W^P}$.

\begin{prop}\label{prop:gp-klande}
The virtual representations $\tilde V_{\mathcal E,P}$ and
$\tilde V_P^{\mathrm{KL}}$ are equal.
\end{prop}

\begin{proof}
The proof is exactly the same as that of Proposition \ref{prop:klande}: once one identifies
the classes of the irreducible and projective objects in the parabolic exotic heart
with the basis $\{g_w^P\}_{w\in W^P}$ and its dual.
\end{proof}

\begin{prop}\label{prop:gp-klandv}
The virtual representations $\tilde V_P^{\mathrm{KL}}$ and $\tilde V_P$ are equal.
\end{prop}

\begin{proof}
Exactly as in the proof of Proposition \ref{prop:klandv}: by
Proposition \ref{prop:parabolic-descent}, the two bases have the same image in the
associated graded for the induced two-sided cell filtration on $N_P$, and the
construction of the lift depends only on that associated-graded basis.
\end{proof}

\begin{theorem}\label{thm:twoliftsP}
Let
\[
V_P=\mathsf{k}[(G/P)(\mathbb{F}_q)].
\]
Then the two virtual $\mathsf{k}$-representations of the algebraic group $G$
attached to $V_P$ by the basis construction of Section \ref{sec:brauer}
and by the coherent-sheaf construction above are equal:
\[
\tilde V_P=\tilde V_{\mathcal E,P}.
\]
\end{theorem}

\begin{proof}
This is immediate from Propositions \ref{prop:gp-klande} and
\ref{prop:gp-klandv}.
\end{proof}

\begin{remark} In Type $A$, the permutation representations $\Ce[G/P(\mathbb{F}_q)]$
span the Grothendieck group of unipotent $G(\mathbb{F}_q)$ representations; this follows
from the corresponding elementary fact about the symmetric group and the bijection
between irreducible unipotent $GL(n,\mathbb{F}_q)$-modules and irreducible representations of
$S_n$.

Thus Theorem \ref{thm:twoliftsP}, together with Theorem \ref{thm:twolifts},
uniquely characterizes the lift appearing in the main construction for all unipotent representations in Type $A$.
\end{remark}

\section{Computations with the elements $M_w$}

In \cite[Section 2]{L}, Lusztig computed by hand the elements $M_w$ in Types $A_1$, $A_2$, $B_2$, $G_2$, and $A_3$. The $M_w$ which we define in the present paper agree with those appearing in loc.\ cit. We reproduce these computations below.

Our definition gives an explicit algorithm for computing $M_w$ in any type. Using our Sage implementation \cite{github}, we obtained explicit expansions for $M_w$ in the basis of Weyl characters in Types $B_3, C_3, A_4, B_4, D_4$. The repository can be used to reproduce all of the computations in the present paper and to work with the constructions introduced here. In Tables \ref{tab:b3}, \ref{tab:c3}, and \ref{tab:a4}, we give such expressions in Types $B_3, C_3$ and $A_4$ to provide some additional examples.

\begin{table}[b]
\caption{A table of $M_w$ for Types $A_1, A_2, B_2, G_2$, previously computed in \cite[Section 2]{L}.}
\renewcommand{\arraystretch}{1.5}
\begin{tabular}{|r|l|}
\hline
\multicolumn{2}{||c||}{Type $A_1$}\\
\hline $w$ & $M_w$\\
\hline
    $1$ & $V_0$ \\
    $s_1$ & $V_{q-1}$\\
    \hline\multicolumn{2}{c}{}\\
    \multicolumn{2}{c}{}\\
    \multicolumn{2}{c}{}\\
    \multicolumn{2}{c}{}\\
    \multicolumn{2}{c}{}\\
    \multicolumn{2}{c}{}\\
\end{tabular} 
\begin{tabular}{|r|l|}
\hline
\multicolumn{2}{||c||}{Type $A_2$}\\
\hline $w$ & $M_w$\\
\hline
    $1$ & $V_{0,0}$ \\
    $s_1$ & $V_{q-1,0}$\\
    $s_2$ & $V_{0,q-1}$\\
    $s_1s_2s_1$ & $V_{q-1,q-1}$\\
    \hline\multicolumn{2}{c}{}\\
    \multicolumn{2}{c}{}\\
    \multicolumn{2}{c}{}\\
    \multicolumn{2}{c}{}\\
\end{tabular}
\begin{tabular}{|r|l|}
\hline
\multicolumn{2}{||c||}{Type $B_2$}\\
\hline $w$ & $M_w$\\
\hline
    $1$ & $V_{0,0}$ \\
    $s_2$ & $V_{0,q-3}$ \\
    $s_1$ & $V_{q-2,0}$ \\
    $s_1s_2s_1$ & $V_{q-1,0}$ \\
    $s_2s_1s_2$ & $V_{0,q-1}$ \\
    $s_2s_1s_2s_1$ & $V_{q-1,q-1}$\\
    \hline\multicolumn{2}{c}{}\\
    \multicolumn{2}{c}{}\\
\end{tabular}
\begin{tabular}{|r|l|}
\hline
\multicolumn{2}{||c||}{Type $G_2$}\\
\hline $w$ & $M_w$\\
\hline
$1$ & $V_{0,0}$ \\
    $s_1$ & $V_{q-4,0}$ \\
    $s_2$ & $V_{0,q-2}$ \\
    $s_1s_2s_1$ & $V_{q-4,1}$ \\
    $s_2s_1s_2$ & $V_{1,q-2}$ \\
    $s_1s_2s_1s_2s_1$ & $V_{q-1,0}$ \\
    $s_2s_1s_2s_1s_2$ & $V_{0,q-1}$ \\
    $s_2s_1s_2s_1s_2s_1$ & $V_{q-1,q-1}$ \\
    \hline
\end{tabular}
\end{table}
\begin{table}
\renewcommand{\arraystretch}{1.5}
\caption{A table of $M_w$ in Type $A_3$, with each $M_w$ being written first in the basis of Weyl modules $V_{\lambda}$ and then in the basis of irreducibles $L_{\lambda}$. The expansion in terms of $V_{\lambda}$ also appears in \cite[Section 2]{L}.}
    \label{tab:a3positivity}
    \centering
    \begin{tabular}{|r|c|c|}
\hline $w$ & $M_w$ in terms of $V_\lambda$ & $M_w$ in terms of $L_\lambda$\\
\hline
    $1$ & $V_{0,0,0}$ & $L_{0,0,0}$ \\
    $s_1$ & $V_{q-1,0,0}$ & $L_{q-1,0,0}$ \\
    $s_3$ & $V_{0,0,q-1}$ & $L_{0,0,q-1}$ \\
    $s_2$ & $V_{0,q-1,0} - V_{0,q-3,0}$ & $L_{0,q-1,0}$ \\
    $s_3s_1$ & $V_{q-1,0,q-1} - V_{q-2,0,q-2}$ & $L_{q-1,0,q-1}$\\
    $s_1s_2s_1$ & $V_{q-1,q-1,0}$ & $L_{q-1,q-1,0}$ \\
    $s_2s_3s_2$ & $V_{0,q-1,q-1}$ & $L_{0,q-1,q-1}$\\
    $s_2s_3s_1s_2$ & $V_{0,q-1,0} + V_{0,q-3,0}$ & $2L_{0,q-3,0} + L_{0,q-1,0}$ \\
    $s_1s_2s_3s_2s_1$ & $V_{q-1,0,q-1} + V_{q-2,0,q-2}$ & $2L_{q-2,0,q-2} + L_{q-1,0,q-1}$\\
    $s_1s_2s_3s_1s_2s_1$ & $V_{q-1,q-1,q-1}$ & $L_{q-1,q-1,q-1}$\\
    \hline
\end{tabular}
\end{table}

\begin{table}
\renewcommand{\arraystretch}{1.5}
\caption{A table of $M_w$ for Type $B_3$.\label{tab:b3}}
\begin{tabular}{|r|l|}
\hline
\multicolumn{2}{||c||}{Type $B_3$}\\
\hline $w$ & $M_w$\\
\hline
    $1$ & $V_{0,0,0}$ \\
$s_3$ & $-V_{0,1,q-5} + V_{1,0,q-3}$ \\
$s_2$ & $V_{0,q-3,0} - V_{0,q-3,2} + V_{1,q-2,0}$ \\
$s_2s_3s_2$ & $V_{1,q-3,0} - V_{0,q-3,2} + V_{0,q-1,0}$ \\
$s_2s_3s_1s_2$ & $-V_{0,q-3,0} - V_{1,q-3,0} + V_{1,q-2,0} + V_{0,q-1,0}$ \\
$s_2s_3s_1s_2s_3s_1s_2s_1$ & $V_{q-1,q-2,0} + V_{q-1,q-1,0}$ \\
$s_3s_2s_3s_1s_2s_3s_2s_1$ & $V_{q-3,0,q-1} + V_{q-2,1,q-3} + V_{q-2,0,q-1} + V_{q-1,0,q-1}$ \\
$s_3s_1s_2s_3s_2s_1$ & $V_{q-2,0,q-3} + V_{q-1,0,q-3} + V_{q-2,1,q-3} + V_{q,0,q-3}$ \\
$s_3s_1s_2s_3s_1$ & $-V_{q-2,0,q-3} - V_{q-3,0,q-1} + V_{q,0,q-3} + V_{q-1,0,q-1}$ \\
$s_1$ & $V_{q-4,0,0} - V_{q-3,0,0}$ \\
$s_3s_1$ & $V_{q-2,0,q-3} - V_{q-2,1,q-3} + V_{q-1,0,q-1}$ \\
$s_1s_2s_3s_2s_1$ & $-V_{q-2,0,0} + V_{q-1,0,0}$ \\
$s_1s_2s_1$ & $-V_{q-1,q-2,0} + V_{q-1,q-1,0}$ \\
$s_1s_2s_3s_1s_2s_1$ & $V_{q-2,q-1,0} + V_{q,q-2,0}$ \\
$s_3s_2s_3s_1s_2s_3s_1s_2$ & $V_{0,q-1,q-1}$ \\
$s_2s_3s_1s_2s_3s_1s_2$ & $V_{0,q-3,0} + V_{0,q-2,0} + V_{0,q-3,2} + V_{0,q-1,0}$ \\
$s_3s_2s_3s_2$ & $V_{0,q-2,q-1}$ \\
$s_3s_2s_3s_1s_2s_3$ & $V_{0,0,q-5} + V_{0,0,q-3} + V_{0,0,q-1}$ \\
$s_3s_2s_3$ & $-V_{0,0,q-5} + V_{0,0,q-1}$ \\
$s_3s_2s_3s_1s_2s_3s_1s_2s_1$ & $V_{q-1,q-1,q-1}$ \\
    \hline
\end{tabular}
\end{table}

\begin{table}
\renewcommand{\arraystretch}{1.5}
\caption{A table of $M_w$ for Type $C_3$.\label{tab:c3}}
\begin{tabular}{|r|l|}
\hline
\multicolumn{2}{||c||}{Type $C_3$}\\
\hline $w$ & $M_w$\\
\hline
    $1$ & $V_{0,0,0}$ \\
$s_3$ & $-V_{0,1,q-3} + V_{1,0,q-2}$ \\
$s_2$ & $V_{0,q-4,0} - V_{1,q-4,1} + V_{2,q-3,0}$ \\
$s_2s_3s_2$ & $V_{2,q-4,0} - V_{1,q-4,1} + V_{0,q-1,0}$ \\
$s_2s_3s_1s_2$ & $-V_{0,q-4,0} + V_{0,q-3,0} - V_{2,q-4,0} - V_{0,q-2,0} + V_{2,q-3,0} + V_{0,q-1,0}$ \\
$s_2s_3s_1s_2s_3s_1s_2s_1$ & $V_{q-1,q-3,0} + V_{q-1,q-2,0} + V_{q-1,q-1,0}$ \\
$s_3s_2s_3s_1s_2s_3s_2s_1$ & $V_{q-3,0,q-1} + V_{q-2,1,q-2} + V_{q-1,0,q-1}$ \\
$s_3s_1s_2s_3s_2s_1$ & $V_{q-2,0,q-2} + V_{q-2,1,q-2} + V_{q,0,q-2}$ \\
$s_3s_1s_2s_3s_1$ & $-V_{q-2,0,q-2} - V_{q-3,0,q-1} + V_{q,0,q-2} + V_{q-1,0,q-1}$ \\
$s_1$ & $V_{q-5,0,0} + V_{q-3,0,0}$ \\
$s_3s_1$ & $V_{q-2,0,q-2} - V_{q-2,1,q-2} + V_{q-1,0,q-1}$ \\
$s_1s_2s_3s_2s_1$ & $V_{q-3,0,0} + V_{q-1,0,0}$ \\
$s_1s_2s_1$ & $-V_{q-1,q-3,0} + V_{q-1,q-1,0}$ \\
$s_1s_2s_3s_1s_2s_1$ & $V_{q-3,q-1,0} + V_{q-1,q-2,0} + V_{q+1,q-3,0}$ \\
$s_3s_2s_3s_1s_2s_3s_1s_2$ & $V_{0,q-1,q-1}$ \\
$s_2s_3s_1s_2s_3s_1s_2$ & $V_{0,q-4,0} + V_{0,q-3,0} + V_{1,q-4,1} + V_{0,q-2,0} + V_{0,q-1,0}$ \\
$s_3s_2s_3s_2$ & $V_{0,q-2,q-1}$ \\
$s_3s_2s_3s_1s_2s_3$ & $V_{0,0,q-3} + V_{0,0,q-1}$ \\
$s_3s_2s_3$ & $-V_{0,0,q-3} + V_{0,0,q-1}$ \\
$s_3s_2s_3s_1s_2s_3s_1s_2s_1$ & $V_{q-1,q-1,q-1}$ \\
    \hline
\end{tabular}
\end{table}

\begin{table}
\renewcommand{\arraystretch}{1.5}
\caption{A table of $M_w$ for Type $A_4$.\label{tab:a4}}
\begin{tabular}{|r|l|}
\hline
\multicolumn{2}{||c||}{Type $A_4$}\\
\hline $w$ & $M_w$\\
\hline
    $1$ & $V_{0,0,0,0}$ \\
    $s_1$ & $V_{q-1,0,0,0}$ \\
    $s_2$ & $V_{0,q-1,0,0} - V_{0,q-3,0,1} + V_{1,q-4,0,0}$ \\
    $s_3$ & $V_{0,0,q-1,0} - V_{1,0,q-3,0} + V_{0,0,q-4,1}$ \\
    $s_4$ & $V_{0,0,0,q-1}$ \\
    $s_3s_1$ & $V_{q-1,0,q-1,0} - V_{q,0,q-3,0} - V_{q-2,0,q-2,1} + V_{q-2,1,q-3,0}$ \\
    $s_4s_1$ & $V_{q-1,0,0,q-1} - V_{q-3,0,0,q-3}$ \\
    $s_4s_2$ & $V_{0,q-3,1,q-2} - V_{0,q-3,0,q} - V_{1,q-2,0,q-2} + V_{0,q-1,0,q-1}$ \\      
    $s_1s_2s_1$ & $V_{q-1,q-1,0,0}$ \\
    $s_3s_4s_3$ & $V_{0,0,q-1,q-1}$ \\
    $s_2s_3s_2$ & $V_{0,q-1,q-1,0} - V_{0,q-2,q-2,0}$ \\
    $s_2s_3s_4s_2s_3s_2$ & $V_{0,q-1,q-1,q-1}$ \\
    $s_1s_2s_3s_1s_2s_1$ & $V_{q-1,q-1,q-1,0}$ \\
    $s_3s_4s_3s_1$ & $V_{q-1,0,q-1,q-1} - V_{q-2,0,q-2,q}$ \\
    $s_2s_3s_1s_2$ & $V_{0,q-1,0,0} + V_{0,q-3,0,1}$ \\
    $s_3s_4s_1s_2s_3s_1$ & $V_{q-1,0,q-1,0} + V_{q,0,q-3,0} + V_{q-2,0,q-2,1} + V_{q-2,1,q-3,0}$ \\
    $s_2s_3s_4s_3s_2$ & $V_{1,q-2,0,q-2} + V_{0,q-1,0,q-1} - V_{0,q-3,0,q} - V_{0,q-3,1,q-2}$ \\
    $s_3s_4s_2s_3$ & $V_{1,0,q-3,0} + V_{0,0,q-1,0}$ \\
$s_2s_3s_4s_1s_2s_3s_1s_2$ & $V_{0,q-2,q-2,0} + V_{0,q-1,q-1,0}$ \\            
$s_2s_3s_4s_3s_1s_2$ & $V_{0,q-3,1,q-2} + V_{0,q-3,0,q} + V_{1,q-2,0,q-2} + V_{0,q-1,0,q-1}$ \\
$s_1s_2s_3s_2s_1$ & $-V_{q-2,1,q-3,0} + V_{q-2,0,q-2,1} - V_{q,0,q-3,0} + V_{q-1,0,q-1,0}$ \\      
$s_1s_2s_3s_4s_3s_2s_1$ & $V_{q-3,0,0,q-3} + 2V_{q-2,0,0,q-2} + V_{q-1,0,0,q-1}$ \\                                 
$s_1s_2s_3s_4s_2s_3s_2s_1$ & $V_{q-2,0,q-1,q-2} + V_{q-2,0,q-2,q} + V_{q-1,0,q-1,q-1}$ \\                             
$s_4s_1s_2s_1$ & $-V_{q,q-2,0,q-2} + V_{q-1,q-1,0,q-1}$ \\   
$s_1s_2s_3s_4s_3s_1s_2s_1$ & $V_{q-2,q-1,0,q-2} + V_{q,q-2,0,q-2} + V_{q-1,q-1,0,q-1}$ \\                            
$s_1s_2s_3s_4s_1s_2s_3s_1s_2s_1$ & $V_{q-1,q-1,q-1,q-1}$ \\
\hline
\end{tabular}
\end{table}

\begin{table}[p]\caption{Table of $\dim(M_w)$ in Type $A_4$.\label{tab:dimtable}}
{\renewcommand{\arraystretch}{1.5}
\begin{tabular}{|r|l|}
\hline
    $w$ & $\dim(M_w)$\\
    \hline
    \hline
    $1$ & $1$\\
    $s_1s_2s_3s_4s_1s_2s_3s_1s_2s_1$ & $q^{10}$\\
    \hline
    $s_2s_3$ & $\frac{1}{4}q + \frac{1}{24}q^2 + \frac{1}{4}q^3 + \frac{11}{24}q^4$\\
    $s_1s_2s_3s_4s_2s_3s_1s_2s_1$ & $\frac{11}{24}q^6 + \frac{1}{4}q^7 + \frac{1}{24}q^8 + \frac{1}{4}q^9$\\
    \hline
    $s_3s_2$ & $\frac{1}{4}q + \frac{1}{24}q^2 + \frac{1}{4}q^3 + \frac{11}{24}q^4$ \\
    $s_1s_2s_3s_4s_1s_2s_3s_2s_1$ & $\frac{11}{24}q^6 + \frac{1}{4}q^7 + \frac{1}{24}q^8 + \frac{1}{4}q^9$\\
    \hline
    $s_4s_1$ & $\frac{1}{9}q^2 + \frac{25}{36}q^4 + \frac{7}{36}q^6$ \\
    $s_2s_3s_4s_1s_2s_3s_1s_2$ & $\frac{7}{36}q^4 + \frac{25}{36}q^6 + \frac{1}{9}q^8$ \\
    \hline
    $s_1s_2s_3s_4$ & $\frac{1}{4}q + \frac{11}{24}q^2 + \frac{1}{4}q^3 + \frac{1}{24}q^4$ \\
    $s_2s_3s_4s_1s_2s_3s_1s_2s_1$ & $\frac{1}{24}q^6 + \frac{1}{4}q^7 + \frac{11}{24}q^8 + \frac{1}{4}q^9$\\
    \hline
    $s_4s_3s_2s_1$ & $\frac{1}{4}q + \frac{11}{24}q^2 + \frac{1}{4}q^3 + \frac{1}{24}q^4$ \\
    $s_1s_2s_3s_4s_1s_2s_3s_1s_2$ & $\frac{1}{24}q^6 + \frac{1}{4}q^7 + \frac{11}{24}q^8 + \frac{1}{4}q^9$\\
    \hline
    $s_3s_4s_1s_2$ & $\frac{1}{12}q^2 + \frac{1}{24}q^3 + \frac{1}{24}q^4 + \frac{11}{24}q^5 + \frac{3}{8}q^6$ \\
    $s_2s_3s_4s_1s_2s_3s_2s_1$ & $\frac{3}{8}q^4 + \frac{11}{24}q^5 + \frac{1}{24}q^6 + \frac{1}{24}q^7 + \frac{1}{12}q^8$\\
    \hline
    $s_4s_2s_3s_1$ & $\frac{1}{12}q^2 + \frac{1}{24}q^3 + \frac{1}{24}q^4 + \frac{11}{24}q^5 + \frac{3}{8}q^6$ \\
    $s_1s_2s_3s_4s_2s_3s_1s_2$ & $\frac{3}{8}q^4 + \frac{11}{24}q^5 + \frac{1}{24}q^6 + \frac{1}{24}q^7 + \frac{1}{12}q^8$\\
    \hline
    $s_4s_2s_3s_2s_1$ & $\frac{7}{36}q^3 + \frac{1}{4}q^4 + \frac{1}{9}q^5 + \frac{1}{4}q^6 + \frac{7}{36}q^7$ \\
    $s_1s_2s_3s_4s_2$& $\frac{7}{36}q^3 + \frac{1}{4}q^4 + \frac{1}{9}q^5 + \frac{1}{4}q^6 + \frac{7}{36}q^7$\\
    \hline
    $s_3s_4s_3s_1s_2s_1$ & $\frac{1}{36}q^4 + \frac{1}{24}q^5 + \frac{1}{9}q^6 + \frac{11}{24}q^7 + \frac{13}{36}q^8$ \\
    $s_2s_3s_4s_1s_2s_3$ & $\frac{13}{36}q^2 + \frac{11}{24}q^3 + \frac{1}{9}q^4 + \frac{1}{24}q^5 + \frac{1}{36}q^6$ \\
    \hline
    $s_1s_2s_3s_4s_3s_1$ & $\frac{1}{36}q^4 + \frac{1}{24}q^5 + \frac{1}{9}q^6 + \frac{11}{24}q^7 + \frac{13}{36}q^8$ \\
    $s_3s_4s_2s_3s_1s_2$ & $\frac{13}{36}q^2 + \frac{11}{24}q^3 + \frac{1}{9}q^4 + \frac{1}{24}q^5 + \frac{1}{36}q^6$\\
    \hline
    $s_3s_4s_2s_3s_1s_2s_1$ & $\frac{1}{18}q^3 + \frac{1}{4}q^4 + \frac{7}{18}q^5 + \frac{1}{4}q^6 + \frac{1}{18}q^7$ \\
    $s_1s_2s_3s_4s_1s_2s_3$ & $\frac{1}{18}q^3 + \frac{1}{4}q^4 + \frac{7}{18}q^5 + \frac{1}{4}q^6 + \frac{1}{18}q^7$ \\
    \hline
    \hline
    
    $s_2s_3s_2$ & $\frac{1}{36}q^3 + \frac{19}{36}q^5 + \frac{4}{9}q^7$ \\
    $s_1s_2s_3s_4s_3s_2s_1$ & $\frac{17}{36}q^3 + \frac{17}{36}q^5 + \frac{1}{18}q^7$ \\
    \hline
\end{tabular}}
    \end{table}
\clearpage

\printbibliography

\end{document}